\documentclass[12pt]{article}
\usepackage{amsfonts}
\usepackage{amssymb}
\usepackage{mathrsfs}
\usepackage{srcltx}
\usepackage{amsmath}
\usepackage{amsthm}
\usepackage{amstext}
\usepackage{amsopn}
\usepackage{graphicx}
\usepackage{subfigure}
\usepackage{grffile}
\usepackage{float}
\usepackage{epsfig}
\usepackage{subfig}
\usepackage{overpic}
\usepackage{epstopdf}
\usepackage{color}
\usepackage{cite}
\usepackage{graphicx}
\usepackage{multirow}
\usepackage[pagewise]{lineno}
\usepackage{enumitem}
\allowdisplaybreaks

\newlist{romanenum}{enumerate}{1}
\setlist[romanenum]{
	label=(\roman*),
	labelwidth=2.5em,
	labelsep*=0.5em,
	leftmargin=*,
	align=left,
	itemindent=0pt,
	parsep=0pt,
	topsep=0pt
}
\graphicspath{{./figs/}}
\DeclareGraphicsExtensions{.eps}

\newtheorem{theorem}{Theorem}[section]
\newtheorem{lemma}{Lemma}[section]
\newtheorem{proposition}{Proposition}[section]

\theoremstyle{definition}
\newtheorem{definition}{Definition}[section]

\newtheorem{remark}{Remark}[section] 

\numberwithin{equation}{section}

\newcommand{\ba}{\begin{array}}
	\newcommand{\ea}{\end{array}}

\begin{document}
	\date{}
	\title{
\bf\large{Bifurcation Analysis  of a Reaction-Diffusion System  with a Cognitive Map  Memory Kernel}\footnote{This work was partially supported by grants from National Natural Science Foundation of China (12371503, 12071382), Natural Science Foundation of Chongqing (CSTB2024NSCQ-MSX0992), Graduate Research and Innovation Project of Southwest University (SWUS25078).}
	}
\author{Jie Liang,\ \ Xiaoli Wang,\footnote{Corresponding Author. Email: wxl711@swu.edu.cn }\ \ Guohong Zhang
		\\
		{\small School of Mathematics and Statistics, Southwest University,
}\\
		{\small Chongqing, 400715, P.R. China.
}}
	
	\maketitle
	
	\noindent
	\begin{abstract}
		{  This paper investigates a single species reaction-diffusion system incorporating a spatiotemporal delay memory kernel, which models the cognitive map of animals, under Neumann boundary conditions. The model can be used to describe the process in which individuals are influenced by historical information during spatial diffusion. An equivalent system construction method with auxiliary variables is introduced to transform the original system into a delay-free coupled reaction-diffusion equation. By employing Fourier modal decomposition and eigenvalue analysis, we conduct stability and bifurcation analyses for both the exponentially decaying weak kernel and the peak type strong kernel, obtaining explicit expressions for the steady state and Hopf bifurcation points. 
Compared with the model in which the memory
term of the continuous-time integral kernel using its own population density, our model exhibits Hopf bifurcations and  steady state bifurcations even under a weak kernel because of  the introduce of a dynamic cognitive map. This implies that a dynamic cognitive map introduces sufficient flexibility to generate both
steady state bifurcations and Hopf bifurcations across a broader range of temporal kernels. Numerical simulations are presented to demonstrate the influence of stable, steady state and Hopf bifurcation regions on the spatiotemporal distribution of solutions. }
		
	\end{abstract}

	\section {Introduction}

The incorporation of memory effects into reaction-diffusion equations has become a established paradigm for modeling population movement and spatial ecology \cite{fanying2023,fanyingkuosan2023,jiyi2019}. While classical Fickian diffusion describes random motion akin to Brownian motion, empirical studies reveal that animal movement in natural environments often exhibits directional biases and path dependencies that deviate from simple random walks \cite{feibulang2010,feibulang2016,feibulang2019}. Many species utilize past experiences to inform dispersal decisions, creating a feedback loop between movement and memory \cite{pianhao2024,renzhimap2013}. This is particularly pronounced in cognitively advanced animals, where memory and learning capabilities fundamentally shape movement strategies \cite{renzhijiyi2023,ewaiODE2016}.


     A foundational approach to model such directed movement is to augment the standard reaction-diffusion equation with an advection term, effectively describing taxis in response to environmental gradients or internal cues \cite{pingliuxiang2007,pingliuxiang2015,pingliuxiang2017}. A significant step in linking this framework to memory was taken by Shi et al. \cite{SWWY2020}, who explicitly introduced a discrete time delay $\tau$
into the diffusion term:
	\begin{equation}\label{eq:1.1}
		\begin{cases}
			\frac{\partial u(x,t)}{\partial t}=d_1 \Delta u(x,t)+d_2 div  (u(x,t)\nabla u(x,t-\tau)) +f(u(x,t)), & x \in \Omega , t>0, \\
			\frac{\partial u(x,t)}{\partial \vec{n}}=0, &  x\in \partial \Omega, t>0,
		\end{cases}
	\end{equation}
where $u(x,t)$ describes the population density at the spatial location $x$ and at time $t$,  $\Omega$ is a bounded domain in $\mathbb{R}^n(n \ge 1)$ with a smooth boundary $\partial \Omega$; $\vec{n}$ is the outward normal vector on the boundary.	Here, $d_1 > 0$ and $d_2 \in \mathbb{R}$ represent the random diffusion rate and the diffusion rate corresponding to memory-based movement, respectively; $\tau$ is the discrete time delay; $f(u)$ denotes the birth or death process.
This model posits that current movement is guided by the population density gradient at a specific past time $t-\tau$. Interestingly, Shi et al. \cite{SWWY2020} demonstrated that the stability of constant steady states in model \eqref{eq:1.1} is independent of the delay magnitude $\tau$, depending only on the diffusion rates $d_1$ and $d_2$.

A key limitation of the discrete delay framework is its assumption that memory is perfectly recalled from a single, specific past moment. Biological memory, however, is typically a continuously weighted integration of past experiences over a period of time. This led to the development of more realistic models employing nonlocal spatiotemporal memory kernels such as
	\[ \frac{\partial u(x,t)}{\partial t}=d_1 \Delta  u(x,t)+d_2  div   (u(x,t)\nabla v(x,t)) +f(u(x,t)), \,  x \in \Omega , t>0, \]
where the memory-driven advection responds to a weighted average of past states \cite{feijubu2015,shikongshizhi2015,SSW2021distributed,feijubu2021,fenbushijian2021}:	
	\[v(x,t)=\left(g \ast \ast u\right) (x,t)=\int_{-\infty}^{t} \int_{\Omega} G(x,y,t-s)g(t-s)u(y,s)  dy ds.\]
	Here,
$v(x,t)$ represents the perceived or remembered environment. The spatial weight function $G(x,y,t-s)$ encodes the probability of the population moving from location $y$ to location $x$ at past time $t-s$, while the temporal weight function $g(t-s)$ weights the influence of past experiences, with more recent events typically carrying greater weight. Formally, $G:\Omega \times \Omega \times \left ( 0,\infty  \right )  \to \mathbb{R}^+ $ satisfies $\int_{\Omega}^{} G(x,y,t) dx=1, \, y\in \Omega,\, t> 0$, and $g:\left [ 0,\infty  \right )  \to \mathbb{R}^+$ meets $\int_{0}^{\infty } g(t) dt=1.$
 Models of this form create mathematical structures rich with nonlocal coupling, drawing interesting parallels to two component systems like the Keller-Segel chemotaxis model.

In a recent review, Wang and Salmaniw\cite{WANG2023zongshu} systematically surveyed the development of memory-based models and outlined some key open challenges in this field. A central concept is the representation of spatial memory as a dynamic cognitive map $a(x,t)$. This map evolves as an organism acquires and updates information about its environment, potentially corresponding to an internal neural representation or an external trace in the landscape \cite{renzhimap2013}. A key open problem  is \lq\lq How might a dynamic cognitive map interact with distributed time delays?\rq\rq\ \cite{WANG2023zongshu}. In an attempt to this open problem,  Liu et al.\cite{LIU2025map}  advanced this framework by incorporating  a spatial nonlocal term into an advective flux and coupling it to an auxiliary ODE governing a \lq\lq memory map\rq\rq. Through comprehensive spectral and stability analysis,  Liu et al.\cite{LIU2025map} provided analytical expressions for bifurcation
values contingent on various model parameters and found that the negative point spectrum with an infinite-dimensional kernel does not affect the stability of the steady state.

The model nonlocal responses in \cite{LIU2025map} relies only on spatial convolution while does not fully capture the temporal evolution of memory. To further address the open problem in \cite{WANG2023zongshu},  
 we  extend the framework of Shi et al.\cite{SSW2021distributed} by explicitly modeling the cognitive map variable
$a(x,t)$ as an independent dynamic component. We dissociate the memory variable
$v(x,t)$  from the instantaneous map
$a(x,t)$:
\begin{equation}\label{eq:1.2}
			v(x,t)=\left(g \ast \ast a\right) (x,t)=\int_{-\infty}^{t} \int_{\Omega} G(x,y,t-s)g(t-s)a(y,s)  dy ds,
	\end{equation}
and the evolution of
$a(x,t)$ is governed by an ordinary differential equation:
\begin{equation}\label{eq:1.3}
		a_t=h(u)-(\mu +\beta u)a, \, x \in \Omega,
	\end{equation}
where $h(u)$ describes the formation rate of the cognitive map,
  $\mu \geq 0$ represents the natural decay rate of memory, and $\beta \geq 0$ is the rate at which revisiting a location erases or suppresses its representation in the map\cite{LIU2025map}.
This formulation yields our  core system:
	\begin{equation}\label{eq:1.4}
		\begin{cases}
			\frac{\partial u(x,t)}{\partial t}=d_1 \Delta  u(x,t)+d_2  div  (u(x,t)\nabla v(x,t)) +f(u(x,t)), & x \in \Omega , t>0, \\
			\frac{\partial a(x,t)}{\partial t}=h(u(x,t))-(\mu +\beta u(x,t))a(x,t),& x \in \Omega , t>0,\\
			\frac{\partial u(x,t)}{\partial \vec{n}}=\frac{\partial a(x,t)}{\partial \vec{n}}=0, &  x\in \partial \Omega, t>0,\\
			u(x,0)=u_0(x),\ \ a(x,t)=\eta (x,t),& x \in \Omega , t\leq 0,
		\end{cases}
	\end{equation}
with $v(x,t)$ equipped with the form \eqref{eq:1.2}.
	
	In this paper, we assume that the spatial weighting function $G(x,y,t)$ in (\ref{eq:1.2}) is the Green's function of diffusion equation with homogeneous Neumann boundary condition and satisfies
	\begin{equation}\label{eq:1.5}
		\begin{cases}
			G_{t}(x, y, t)=d_{1} \Delta x G(x, y, t), & x \in \Omega, t>0 , \\
			\frac{\partial  G(x, y, t)}{\partial  \vec{n}}=0 , & x \in \partial \Omega,t>0 ,\\
			G(x, y, 0)=\delta(x-y), &
		\end{cases}
	\end{equation}
	where $\delta(x)$	is the Dirac delta function on $\Omega$. Then $G(x,y,t)$ has the following form
\begin{equation}\label{gamma}
		 G(x, y, t)  =\sum_{n=0}^{\infty} e^{-d_{1} \lambda_n t} \phi_{n}(x) \phi_{n}(y),
\end{equation}
	where $\lambda_{n} $ is the $n-$th  eigenvalues of the eigenvalue problem
	\begin{equation}\label{eq:1.6}
		\begin{cases}
		-\Delta \phi(x)=\lambda \phi(x), & x \in \Omega , \\
		\frac{\partial \phi(x)}{\partial \vec{n}}=0, & x \in \partial \Omega ,
		\end{cases}
	\end{equation}
satisfying $0=\lambda_{0}<\lambda_{1} \leqslant \lambda_{2} \leqslant \cdots \leqslant \lambda_{n} \leqslant \cdots \rightarrow+\infty $, as  $n \rightarrow \infty$,	and $\phi_{n}(x)$ is the eigenfunction corresponding to  $\lambda_{n} $.
	For the temporal weighting function $g(t)$, we choose a Gamma distribution function of order $k$ \cite{MacDonald}:
	\begin{equation}\label{eq:1.7}
		g_k(t)=\frac{t^k e^{-\frac{t}{\tau}} }{\tau^{k+1}k!}.
	\end{equation}
	Typically, we consider the cases of \( k = 0 \) (weak kernel) and \( k = 1 \) (strong kernel).
	The weak kernel function \( g_0(t) \) is strictly decreasing with respect to \( t \), describing a process of memory decay over time.
	In the strong kernel case, the function \( g_1(t) \) attains its maximum at \( t = \tau \), which  increases monotonically on the interval \( (0, \tau) \), and decreases monotonically on the interval \( (\tau, \infty) \).
	This describes two key ecological scenarios: initial knowledge acquisition followed by memory decay.
	The parameters \( k \) and \( \tau \) determine the memory range and the average time delay, respectively.
	The mean and variance of \( g_k(t) \) are given by $({k+1}){\tau}$ and $({k+1}){\tau^2}$, respectively, where \( \tau \) is  the average delay. Therefore, we take $\tau$ as a parameter to measure the impact of spatial memory on the dynamics.



The primary objective of this paper is to conduct a rigorous analysis of  system \eqref{eq:1.4}, focusing on how the shape of the temporal memory kernel, specifically the exponentially decaying \lq\lq weak\rq\rq\  kernel ($k=0$) and the unimodal \lq\lq strong\rq\rq\  kernel ($k=1$), influences stability, bifurcations, and  pattern formation. A primary analytical difficulty of the model lies in the memory term
$v(x,t)$, which involves a distributed delay integral over the entire past, introducing a nonlocality in time that prevents the direct application of classical dynamical systems tools for stability and bifurcation analysis.
To overcome this obstacle and systematically investigate the dynamical effects of the memory kernel, a central strategy adopted in this work is to introduce auxiliary variables that transform the original distributed delay system into an equivalent delay-free coupled reaction-diffusion system. This transformation not only places the problem within the classical framework of partial differential equations, allowing the use of eigenvalue analysis, stability criteria, and bifurcation theory, but also elucidates how the additional degrees of freedom introduced by the memory process influence the system dynamics.

Therefore, in Section 2 that follows, we will present the  equivalent systems for both the weak and strong kernel cases, and  show their dynamical equivalence to the original system (1.4) with respect to steady states, periodic solutions, and other dynamical invariants. This formulation serves as the mathematical foundation for the linear stability analysis, Turing and Hopf bifurcation analysis in Section 3, and the numerical simulations in Section 4. Section 5 provides conclusions and discussions.

	
For the subsequent analysis, $\Omega$ denotes a bounded domain in $\mathbb{R}^N(N=1, 2, 3)$ with a smooth boundary $\partial \Omega$. We define the function spaces $X=\left\{u\in W^{2,p}\left(\Omega\right):\frac{\partial u}{\partial \vec{n}}=0, x \in \partial \Omega\right\}$ and $Y=L^p(\Omega)$. Let $\mathbb{N}$ represent the set of positive integers and $\mathbb{N}_0=\mathbb{N}\cup \left\{0\right\}$ the set of non‑negative integers.
 We assume that the functions $f$ and
	$h$ satisfy
	\[ \begin{array}{ll}
		\text{(H1)} & f(u) \in C^3([0,\infty)), f(0) = f(1) = 0, f'(0) > 0, f'(1) < 0, f(u) > 0 \ \text{for}\ u \in (0,1), \\
		& \text{and}\ f(u) < 0 \ \text{for}\ u > 1.  \\
		\text{(H2)} & h(u) \in C^3([0,\infty)),\ h(u) > 0\ \text{on}\ (0,\infty),\ h(0) = 0. \\
	\end{array} \]

	\section{Equivalent systems}
	In this section, we derive the equivalent formulations for the reaction-diffusion system (\ref{eq:1.4}) incorporating the memory variable $v(x,t)$ defined in \eqref{eq:1.2}, where the spatial kernel $G$ takes the form \eqref{gamma} and the temporal kernel $g$ is specified by \eqref{eq:1.7} in its weak and strong forms. The introduction of auxiliary variables effectively “unpacks” the spatio-temporal memory convolution, replacing the historically dependent term
$v(x,t)$ with an instantaneous variable that satisfies a diffusion equation coupled to the cognitive map. This transformation converts the distributed delay  into a system of differential equations without explicit delay, allowing us to analyze stability and bifurcation using standard spectral methods for reaction-diffusion systems.

	Firstly,  similarly to \cite{LiuWZ2025dengjiaxitong}, we give another definition for Green's function of the
	following diffusion equation
	\begin{equation}\label{eq:2.1}
		\begin{cases}
			 \mathcal{L} \omega(x,t) := \frac{\partial \omega(x,t)}{\partial t} - d_1  \Delta \omega(x,t) = 0, \quad & x \in \Omega, \ t > 0,  \\
			\frac{\partial \omega(x,t)}{\partial \vec{n}} = 0, \quad & x \in \partial \Omega, \ t > 0.
		\end{cases}	
	\end{equation}

	\begin{definition}\label{de21}\cite{evans2010pde,Friedman1983PartialDE}
		Given $y\in \Omega$, $G(x,y,t)$ is the Green's function of (2.1) for $(x,t) \in \Omega \times (0,\infty)$, if $	\omega(x,t) := \int_{\Omega} G(x,y,t)f(y)  dy$ is the solution of
		\[
		\mathcal{L} \omega(x,t)= 0, \quad (x,t) \in \Omega \times (0,\infty),
		\]
		for any function $f$ with compact support in $\Omega$, and satisfies
		\begin{equation*}
      \begin{aligned}
		&\frac{\partial \omega(x,t)}{\partial \vec{n}} = 0, \quad (x,t) \in \partial \Omega \times (0,\infty),\\
		&\lim_{t \to 0} \omega(x,t) = f(x), \quad x \in \Omega.
       \end{aligned}
		\end{equation*}
	\end{definition}
	
	The following two propositions provide equivalent systems of  (\ref{eq:1.4}) for the two types of temporal kernels.
	
    \begin{proposition}\label{p1}
     (Weak kernel)	Suppose that the distributed memory kernel $g(t)$ is given by the weak kernel function $g(t) = g_{0}(t) = \frac{1}{\tau} e^{-\frac{t}{\tau}}$, and define
    	\begin{equation}\label{v0}
    	v(x,t)=v_0(x,t): = (g_{0} \ast\ast a)(x,t) = \int_{-\infty}^{t} \int_{\Omega} G(x,y,t-s) g_{0}(t-s) a(y,s)  dy  ds.
    	\end{equation}
    	Then $(u,a)$ is a solution of $(\ref{eq:1.4})$ if and only if $(u,v,a)$ is a solution of
    	\begin{equation}\label{eq:2.2}
    		\begin{cases}
    			 u_{t} = d_{1} \Delta u + d_{2}  div (u \nabla v) + f(u) , & x \in \Omega, \ t > 0, \\
    			 v_{t} = d_{1} \Delta v + \frac{1}{\tau} (a - v), & x \in \Omega, \ t > 0, \\
    			 a_{t} = h(u) - (\mu + \beta u) a, & x \in \Omega, \ t > 0, \\
    			 \frac{\partial u}{\partial \vec{n}} = \frac{\partial v}{\partial \vec{n}} =0, & x \in \partial \Omega, \ t > 0, \\
    			 a(x,t) = \eta(x,t), & x \in \Omega, \ t \in (-\infty,0], \\
    			 u(x,0)=u_0(x), & x \in \Omega,\\
    			 v(x,0) = \frac{1}{\tau} \int_{-\infty}^{0} \int_{\Omega} G(x,y,-s)e^{\frac{s}{\tau}} \eta(y,s) dy  ds, & x \in \Omega.
    		\end{cases}
    	\end{equation}
    	
    	Moreover, $(u,a)$ is a steady state  of $(\ref{eq:1.4})$  if and only if $(u,v,a)$ is a steady state of $(\ref{eq:2.2})$; $(u,a)$ is a periodic solution of $(\ref{eq:1.4})$ with period $T$ if and only if  $(u,v,a)$ is a periodic solution of $(\ref{eq:2.2})$ with period $T$.
    \end{proposition}

    \begin{proof}
    	 See Appendix A.
    \end{proof}

     \begin{proposition}\label{p2}(Strong kernel)
    	Suppose $g(t)$ is given by the strong kernel function $g(t) = g_{1}(t) = \frac{t}{\tau^2} e^{-\frac{t}{\tau}}$, and define
    	\[
    	v(x,t) =v_1(x,t):=  (g_{1} \ast\ast a)(x,t) := \int_{-\infty}^{t} \int_{\Omega} G(x,y,t-s) \, g_{1}(t-s) \, a(y,s)  dyds.
    	\]
    	Then $(u,a)$ is the solution of $(\ref{eq:1.4})$  if and only if $(u,v,w,a)$ is the solution of
    	\begin{align}\label{eq:2.11}
    		\begin{cases}
    			 u_{t} = d_{1} \Delta u + d_{2}  div  (u \nabla v) + f(u) , & x \in \Omega, \ t > 0, \\
    			 v_{t} = d_{1} \Delta v + \frac{1}{\tau} (w - v), & x \in \Omega, \ t > 0, \\
    			 w_{t} = d_{1} \Delta w + \frac{1}{\tau} (a - w), & x \in \Omega, \ t > 0, \\
    			 a_{t} = h(u) - (\mu + \beta u) a, & x \in \Omega, \ t > 0, \\
    			 \frac{\partial u}{\partial \vec{n}} = \frac{\partial v}{\partial \vec{n}} = \frac{\partial w}{\partial \vec{n}} = 0, & x \in \partial \Omega, \ t > 0, \\
    			 a(x,t) = \eta(x,t), & x \in \Omega, \ t \in (-\infty,0], \\
    			 u(x,0) = u_{0}(x), & x \in \Omega, \\
    			 v(x,0) = -\frac{1}{\tau^2} \int_{-\infty}^{0} \int_{\Omega} G(x,y,-s) \, s \, e^{\frac{s}{\tau}} \, \eta(y,s) dy  ds, & x \in \Omega, \\
    			 w(x,0) = \frac{1}{\tau} \int_{-\infty}^{0} \int_{\Omega} G(x,y,-s) \, e^{\frac{s}{\tau}} \, \eta(y,s)  dy  ds, & x \in \Omega.
    		\end{cases}
    	\end{align}

    Moreover, $(u,a)$ is a steady state  of $(\ref{eq:1.4})$  if and only if $(u,v, w, a)$ is a steady state of $(\ref{eq:2.11})$; $(u,a)$ is a periodic solution of $(\ref{eq:1.4})$ with period $T$ if and only if  $(u,v, w, a)$ is a periodic solution of $(\ref{eq:2.11})$ with period $T$.
    \end{proposition}

\begin{proof}
   	 The proof is similar to that of Proposition \ref{p1}.
 \end{proof}

	\section{Stability and Bifurcation Analysis}

	\subsection{Weak kernel case}
	In this subsection, we will study the stability of the positive constant steady state  and the associated bifurcations of (\ref{eq:1.4}) with weak temporal kernel.
By virtue of Proposition \ref{p1}, the stability analysis of system (\ref{eq:1.4}) with weak temporal kernel can be carried out via its equivalent delay-free system (\ref{eq:2.2}).

	Under the assumptions (H1) and (H2), system (\ref{eq:2.2}) admits a constant positive steady state $\left( 1, \theta,\theta\right) $, where $\theta=\frac{h(1)}{\mu+\beta}$.
Linearizing equation (\ref{eq:2.2}) at the constant positive steady state $(1,\theta,\theta )$ yields
	\begin{equation}\label{eq:3.1}
		\begin{cases}
			\frac{\partial \tilde u(x,t)}{\partial t}=d_1 \Delta  \tilde u(x,t)+d_2\Delta \tilde v(x,t) +f'(1)\tilde u(x,t), & x \in \Omega , t>0, \\
			\frac{\partial \tilde{v}(x,t)}{\partial t}=d_1 \Delta \tilde{v}(x,t)+\frac{1}{\tau}\left(\tilde{a}(x,t)-\tilde{v}(x,t)\right), &x \in\Omega,t>0,\\
			\frac{\partial \tilde a(x,t)}{\partial t}=h_1\tilde u(x,t)+h_2\tilde a(x,t),& x \in \Omega , t>0,\\
			\frac{\partial \tilde u(x,t)}{\partial \vec{n}}=\frac{\partial \tilde v(x,t)}{\partial \vec{n}}=0, &  x\in \partial \Omega, t>0,
		\end{cases}
	\end{equation}
	where
\begin{equation}\label{h}
 h_1:=h'(1)-\beta \theta,\   h_2:=-(\mu +\beta )<0.
\end{equation}
	The stability of $\left( 1, \theta,\theta\right)$  is determined by the following eigenvalue problem
		\begin{equation}\label{eq:3.2}
		\begin{cases}
			d_1 \Delta \phi +d_{2}\Delta \psi +f'(1) \phi =\sigma \phi, & x \in \Omega , \\
			d_{1}\Delta  \psi +\frac{1}{\tau}(\varphi-\psi)=\sigma \psi, & x \in \Omega ,\\
			h_1 \phi+h_2 \varphi=\sigma \varphi , & x \in \Omega ,\\
			\frac{\partial \phi}{\partial \vec{n}}=\frac{\partial \psi}{\partial \vec{n}}=\frac{\partial \varphi}{\partial \vec{n}}=0,  & x \in \partial \Omega.
		\end{cases}
	\end{equation}
	The eigenvalues of Eq.(\ref{eq:3.2}) are precisely the eigenvalues of the Jacobian matrix
	\[J_n^w= \begin{pmatrix}
		-d_{1} \lambda_{n}+f'(1) & -d_2 \lambda_{n}  & 0\\
		0 & -d_{1} \lambda_{n}-\frac{1}{\tau}  & \frac{1}{\tau}\\
		h_1 & 0 & h_2
	\end{pmatrix},\,  n\in \mathbb{N}_0.\]
	Hence we obtain the characteristic equation
	\begin{equation}\label{eq:3.3}
		\sigma^3+A_n(d_2)\sigma^2+B_n(d_2)\sigma+C_n(d_2)=0, \, n\in \mathbb{N}_0 ,
	 \end{equation}	
	with
	\begin{equation*}
		\begin{aligned}\label{eq:3.4}
			&A_n(d_2)=2d_1 \lambda_n -f'(1)-h_2+\frac{1}{\tau},\\
			&B_n(d_2)=-d_1 \lambda_n h_2-\frac{h_2}{\tau}+\left( d_1 \lambda_n -f'(1)\right) \left( d_1 \lambda_n -h_2+\frac{1}{\tau}\right), \\
			&C_n(d_2)=\ {-d_1}^2 {\lambda_n}^2 h_2 +\lambda_n\left( -\frac{d_1 h_2}{\tau}+f'(1)d_1 h_2+\frac{d_2 h_1}{\tau}\right) +\frac{f'(1)h_2}{\tau}.
	   \end{aligned}
	\end{equation*}
 By the Routh-Hurwitz  criterion, the matrix $J_n^w$ is stable if and only if
	\begin{equation}\label{eq:3.5}
		A_n(d_2)>0 , Q_n(d_2):=A_n(d_2)B_n(d_2)-C_n(d_2)>0 , C_n(d_2)>0
		\end{equation}
for all $ n\in \mathbb{N}_0$. Note that for $n=0 (\lambda_0=0)$, the three eigenvalues of $J_0^w$ are $f'(1), -\frac{1}{\tau}$ and $h_2$, which are all negative. Hence we only need to consider the case $n\geq 1.$

For $ n\in \mathbb{N}$,	it follows from  (H1)  that
	\[ A_n(d_2)=2d_1 \lambda_n -f'(1)-h_2+\frac{1}{\tau} >0, B_n(d_2)>0.\]
	Consequently, from the definition in \eqref{eq:3.5} $C_n(d_2)=0$ and $Q_n(d_2)=0$ cannot hold simultaneously.
	Therefore, the matrix $J_n^w$ may lose its stability either when there exists  $d_2$  such that
	$C_n(d_2)=0$ for some $n\in \mathbb{N}$, in which case  $J_n^w$ exhibits a zero eigenvalue, or when there exists  $d_2$  such that $Q_n(d_2)=0$ for some $n\in \mathbb{N}$, corresponding to a pair of purely imaginary eigenvalues $\pm \sqrt{B_n(d_2)}i$.
 To analyze these two instability scenarios,  we define the following functions
	\begin{equation*}
		\begin{aligned}\label{eq:3.6}
			&A(d_2, p):=2d_1 p -f'(1)-h_2+\frac{1}{\tau},\\
			&B(d_2, p):=-d_1 h_2 p-\frac{h_2}{\tau}+\left( d_1 p -f'(1)\right) \left( d_1 p -h_2+\frac{1}{\tau}\right) ,\\
			&C(d_2, p):=\left( {-d_1}^2 h_2\right) {p}^2+\left( -\frac{d_1 h_2}{\tau}+f'(1)d_1 h_2+\frac{d_2 h_1}{\tau}\right)p +\frac{f'(1)h_2}{\tau},\\
			&Q(d_2, p):=A(d_2, p)B(d_2, p)-C(d_2, p)=2{d_1}^3 p^3+a_1p^2+b_1p+c_1,
		\end{aligned}
	\end{equation*}
for $d_2\in \mathbb{R}, p> 0$, where
	\begin{equation*}
		\begin{aligned}\label{eq:3.7}
			&a_1=\left( \frac{3}{\tau}-4h_2-3f'(1)\right) {d_1}^2 ,\\
			&b_1=\left( \frac{1}{{\tau}^2}+\frac{-4\left(f'(1)+h_2\right)}{\tau}+4f'(1)h_2+{\left(f'(1)\right)^2}+2{h_2}^2\right) d_1-\frac{h_1d_2}{\tau},\\
			&c_1=-\frac{f'(1)+h_2}{\tau^2}+\frac{{\left( f'(1)+h_2\right) }^2}{\tau}-{\left( f'(1)\right) }^2h_2-f'(1){h_2}^2.
		\end{aligned}
	\end{equation*}
	Solving $C(d_2,p)=0$ for $d_2$ gives
	\begin{equation}\label{eq:3.8}
	   d_2^S(p)=\frac{h_2\left( d_1p-f'(1)\right) \left( d_1\tau p+1\right) }{h_1p}.
	\end{equation}	
	Solving $Q(d_2, p)=0$ for $d_2$ yields
	\begin{equation}\label{eq:3.9}
	   d_2^H(p)=\frac{(1+2d_1\tau p-\tau f'(1))(h_2\tau-1 -d_1\tau p)(h_2-d_1p+f'(1))}{\tau h_1p}.
	\end{equation}	
	
	We impose an additional hypothesis:\\
   \noindent
   \text{(H3)} \quad $h'(1) > \beta\theta.$


	The following lemma summarizes the basic properties of $d_2^S(p)$ and $d_2^H(p)$.
	\begin{lemma}\label{lem3.1}
		Assume that $(H1)-(H3)$ are satisfied. Then  the following statements hold.
		\begin{romanenum}
			\item For $d_2^S(p)$ defined in $(\ref{eq:3.8})$, there exists $p_*>0$ such that $d_2^S(p)$ is increasing for $p \in \left ( 0,p_* \right )  $ and decreasing for $p \in  (p_*,\infty)$. Moreover, $d_2^S(p)$ attains its global maximum value $d_{2,S}^*<0$ at $p=p_*$, and $\displaystyle\lim_{p \to 0} d_2^S(p)=-\infty $, $\displaystyle\lim_{p \to +\infty} d_2^S(p)=-\infty $.
		    \item For $d_2^H(p)$ defined in $(\ref{eq:3.9})$, there exists $p^*>0$ such that $d_2^H(p)$ is decreasing for $p \in  ( 0,p^* )  $ and increasing for $p \in  (p^*,\infty)$. Moreover, $d_2^H(p)$ attains its global minimum $d_{2,H}^*>0$ at $p=p^*$, and  $\displaystyle\lim_{p \to 0} d_2^H(p)=+\infty $, $\displaystyle\lim_{p \to +\infty} d_2^H(p)=+\infty$.
		\end{romanenum}
		
	\end{lemma}
 \begin{proof}
    	 See Appendix B.
    \end{proof}

	For $n\in \mathbb{N}$, define
	\begin{equation}\label{eq:3.10}
	d_{2,n}^S:=d_2^S(\lambda_n),\ \ \ d_{2,n}^H:=d_2^H(\lambda_n),
	\end{equation}
	 where $d_2^S(p)$ and $d_2^H(p)$ are defined in (\ref{eq:3.8}) and (\ref{eq:3.9}).
	Then $\sigma=0$ is a root of (\ref{eq:3.3}) when $d_2=d_{2,n}^S$ and $\sigma=\pm i\omega _0:= \pm\sqrt{B_n(d_{2,n}^H)}i$ is a pair of purely imaginary roots of (\ref{eq:3.3}).
Let
\begin{equation}\label{eq:3.12}
			d_{2,N}^S= \max _{n\in \mathbb{N} } \left \{ d_{2,n}^S \right \}, \ \ \, d_{2,M}^H= \min _{n\in \mathbb{N} } \left \{ d_{2,n}^H \right \}.
		\end{equation}	
	 About the stability of the constant steady state $(1,\theta,\theta)$ of system (\ref{eq:2.2}) we have the following results.
	\begin{theorem}\label{3.2}
		Assume that $d_1>0$, $\tau>0$, and $(H1)-(H3)$ hold. Let $d_{2,N}^S$ and $d_{2,M}^H$ be defined as in $(\ref{eq:3.12})$.
		Then the constant equilibrium
		$(1,\theta,\theta)$ of system $(\ref{eq:2.2})$ is locally asymptotically stable when $d_{2,N}^S<d_2<d_{2,M}^H$, and unstable if $d_2<d_{2,N}^S$ or  $d_2>d_{2,M}^H$.
		 \end{theorem}
		\begin{proof}
			Obviously, from Lemma \ref{lem3.1}, $d_{2,N}^S$ and $d_{2,M}^H$ exist and $d_{2,N}^S<0<d_{2,M}^H$. Since $C_n(d_2)$ increases monotonically with respect to $d_2$, so when $d_2>d_{2,N}^S$,  we have $C_n(d_2)>0$ for all $n\in \mathbb{N}$. Similarly, $Q_n(d_2)$ decreases monotonically with respect to $d_2$, so when $d_2<d_{2,M}^H$, we have $Q_n(d_2)>0$ for all $n\in \mathbb{N}$.
			Consequently, by the Routh-Hurwitz conditions (\ref{eq:3.5}), all the eigenvalue of $J_n^w$ have negative real parts for $n\in \mathbb{N}$. Note that all of the three eigenvalues $J_0^w$ are negative. Therefore all eigenvalues of $J_n^w$ have negative real parts for $n\in \mathbb{N}_0$, which implies that $(1,\theta,\theta)$ is locally asymptotically stable.
			
			When $d_2<d_{2,N}^S$, we have $A_n(d_2)>0$, $C_n(d_2)<0$, $Q_n(d_2)>0$; When $d_2>d_{2,M}^H$, we have $A_n(d_2)>0$, $C_n(d_2)>0$,  $Q_n(d_2)<0$. By the Routh-Hurwitz
stability criterion,  the matrix $J_n^w$ has at least one eigenvalue with positive real part in these two cases, which implies that $(1,\theta,\theta)$ of system $(\ref{eq:2.2})$ is unstable. This completes the proof.
		\end{proof}
		
		By direct calculation, one can get the following transversality condition  at $d_2=d_{2,n}^H$.		
		\begin{lemma}\label{lem3.3}
			Let $d_{2,n}^H$ defined as in $(\ref{eq:3.10})$. At $d_2=d_{2,n}^H$, the Eq.$(\ref{eq:3.3})$ has a pair of complex conjugate roots $\pm i\omega_0$. Moreover, $\alpha'(d_{2,n}^H)>0$ where $\sigma=\alpha(d_2)\pm i\omega(d_2)$  denotes the continuation of these roots.
		\end{lemma}	
	    \begin{proof}
	    	Clearly, $\alpha({d}_{2,n}^H)=0$. Differentiating both sides of (\ref{eq:3.3}) with respect to $d_2$
 and evaluating at $d_2 ={d}_{2,n}^H$ with $\sigma = i\omega_0$, and noting that  ${\omega_0}^2=B_n({d}_{2,n}^H)$, we have
	    		\[ \left. \frac{d\sigma}{d(d_2)} \right|_{d_2 = {d}_{2,n}^H}= \frac{h_1 \lambda_n}{2\tau}\cdot \frac{ \omega_0^2+ iA_n({d}_{2,n}^H)\omega_0}{ (\omega_0^2)^2 + (A_n({d}_{2,n}^H)\omega_0)^2 }. \]
	    	Therefore,
	    	\[
	    	\alpha'({d}_{2,n}^H)
	    	= \mathrm{Re} \left[ \left. \frac{d\sigma}{d(d_2)} \right|_{d_2 = {d}_{2,n}^H} \right]
	    	= \frac{ h_1\lambda_n}{2\tau \left[  \omega_0^2 + (A_n({d}_{2,n}^H))^2  \right]} > 0.
	    	\]
	    	This completes the proof.
	    \end{proof}

	Applying the Hopf bifurcation theorem for quasilinear reaction-diffusion systems \cite{A1991hopffenzhi}, we obtain the following results.
	
	    \begin{theorem}\label{3.4}
	    	Suppose that $\lambda_n$ is a simple eigenvalue of $(\ref{eq:1.6})$, $B_n$, $d_{2,n}^H$ are given by $(\ref{eq:3.4})$ and $(\ref{eq:3.10})$ respectively, and $d_{2,n}^H\ne d_{2,k}^H$ for any $k\in \mathbb{N}$ and $k\ne n$. Then a Hopf bifurcation occurs at $d_2=d_{2,n}^H$ for $(\ref{eq:2.2})$, and there exists a family of periodic orbits of the following form:
	    	\[ \left\{\left( U_n(x,t,s),T_n(s),d_2^{(n)}(s)\right) :s\in (0,\delta) \right\} .\]
	    	Here,  $U_n(x,t,s)=\left( u_n(x,t,s),v_n(x,t,s),a_n(x,t,s)\right) $ is a time-periodic solution of equation $(\ref{eq:2.2})$ with period $T_n(s)$ when $d_2=d_2^{(n)}(s)$, satisfying
	    	\[ d_2^{(n)}(0)=d_{2,n}^H ,\, \lim_{s \to 0} U_n(x,t,s)=(1,\theta,\theta), \, \lim_{s \to 0}T_n(s)=\frac{2\pi}{\sqrt{B_n}} .\]
	    \end{theorem}
	
	    Next, we use the bifurcation theory of simple eigenvalues\cite{Crandall1970Bifurcation} to prove that  system (\ref{eq:2.2}) undergoes a steady-state bifurcation and show the existence of nonconstant steady states. The steady-state solutions of (\ref{eq:2.2}) satisfy
	    \begin{equation}\label{eq:3.15}
	    	\begin{cases}
	    		d_1 \Delta u(x)+d_{2} div(u(x)\nabla v(x))+f(u(x))=0 , & x \in \Omega ,\\
	    	    d_{1} \Delta v(x)+\frac{1}{\tau}(a(x)-v(x))=0, & x \in \Omega, \\
	    		h(u(x))-(\mu +\beta u(x))a(x) =0, & x \in \Omega , \\
	    		\frac{\partial u}{\partial \vec{n}}=\frac{\partial v}{\partial \vec{n}}=0,  & x \in \partial \Omega,
	    	\end{cases}
	    \end{equation}
	    and $(1,\theta,\theta)$ is a constant steady state.
	    \begin{theorem}\label{3.5}
	    	Suppose that $d_1$, $d_2$, $f$, $h$ satisfy conditions $(H1)-(H3)$, $\tau>0$ and let $d_{2,n}^S$ be defined as in $(\ref{eq:3.10})$. Then the  following statements hold.
	    	\begin{romanenum}
	    		\item Suppose that $d_{2,n}^S\ne d_{2,k}^S$ for any $ k\in \mathbb{N}$ and $k\ne n$, and $\lambda_n$ is a  simple eigenvalue of $(\ref{eq:1.6})$. Then $d_2=d_{2,n}^S$ is a steady-state bifurcation point for $(\ref{eq:3.15})$. Near $\left( d_{2,n}^S,1,\theta,\theta\right)$,  system $(\ref{eq:2.2})$ possesses a line of homogeneous solutions \[ \Gamma _0:=\left\{(d_2,1,\theta,\theta): \, d_2\in \mathbb{R}\right\} \] and a smooth curve $\Gamma_n$ bifurcating from $\Gamma_0$ at $d_2=d_{2,n}^S$ in a form of \[ \Gamma_n=\left\{\left( d_{2,n}(s),U_n(s,x),V_n(s,x),A_n(s,x)\right) :-\delta<s<\delta\right\},\]
	    		where
	    		\[ \begin{aligned}
	    			&U_n(s,x)=1+s\phi_n(x)+sz_{1,n}(s,x),\\
	    			&V_n(s,x)=\theta-\frac{sh_1\phi_n(x)}{h_2\left(d_1\lambda_n\tau+1\right)}+sz_{2,n}(s,x),\\
	    			&A_n(s,x)=\theta-\frac{sh_1\phi_n(x)}{h_2}+sz_{3,n}(s,x).
	    		\end{aligned} \]
	    		Here, $ d_{2,n}(s), z_{1,n}(s,x), z_{2,n}(s,x), z_{3,n}(s,x)$ are smooth functions defined for $-\delta<s<\delta$ such that $d_{2,n}(0)=d_{2,n}^S$, $z_{i,n}(0,x)=0 (i=1,2,3)$ and $\delta$ is a positive constant.
	    		\item Let $\Omega=(0,l\pi)$, we have $\lambda_n=\frac{n^2}{l^2}$ and $\phi_n=cos\left(\frac{nx}{l}\right),  n\in \mathbb{N}$. In this case $d_{2,n}^\prime(0)=0$ and
	    	\[ d_{2, n}^{\prime \prime}(0) = \frac{m_1}{4 \lambda_n \alpha_n} . \]
	    	Here,
	    	\begin{equation}
	    		\begin{aligned}\label{eq:3.16}
	    			&\begin{aligned}
	    				m_1=&f^{\prime \prime \prime}(1)-8 d_{2,n}^S \lambda_n \Theta_{2}^{2} + 4 d_{2,n}^S \lambda_n \alpha_n\Theta_{1}^{2} + 4(f^{\prime \prime}(1)-2 d_{2, n}^{S} \lambda_n \alpha_n) \Theta_{1}^{1}+2f^{\prime \prime}(1) \Theta_{1}^{2}\\
	    			& + \frac{2\beta h_1(2 \Theta_{3}^{1}+\Theta_{3}^{2})}{h_2} -\frac{2\beta h_1^2(2 \Theta_{1}^{1}+\Theta_{1}^{2})}{h_2^2},
	    		\end{aligned}\\
	    			&\alpha_n=-\frac{h_1}{h_2\left(d_1\lambda_n\tau+1\right)}.
	    		\end{aligned}
	    	\end{equation}
	    	and   $\Theta_{1}^{1}, \Theta_{1}^{2}, \Theta_{2}^{1}, \Theta_{2}^{2}, \Theta_{3}^{1}, \Theta_{3}^{2}  $ are given by
	    	
	    		\begin{equation}
	    			\begin{split}\label{eq:3.17}
	    					&\Theta_{1}^1=-\frac{f^{\prime \prime}(1)}{2 f^{\prime}(1)}, \,  \Theta_{2}^1=\Theta_{3}^1=\frac{h_{1}h_2 f^{\prime \prime}(1)-2 f^{\prime}(1) \beta h_1 }{2 f^{\prime}(1) h_{2}^2}, \\
	    					&\Theta_{1}^{2}=\frac{-\left[\left(\frac{f^{\prime \prime}(1)}{2}-2 d_{2} \alpha_n \lambda_{n}\right)\left(-4 d_{1} \lambda_{n}\tau-1\right) h_{2}^2-4\beta h_1d_{2} \lambda_{n}\right]}{\left(-4 d_{1} \lambda_n +f^{\prime}(1)\right)\left(-4 d_{1} \lambda_n\tau -1\right) h_{2}^2- 4 d_{2} \lambda_n h_{1}h_2}, \\
	    					&\Theta_{2}^{2}=\frac{ -h_{1}h_2\left(\frac{f^{\prime \prime}(1)}{2}-2 d_{2} \alpha_n \lambda_n \right)+\beta h_1 \left(-4 d_{1} \lambda_n +f^{\prime}(1)\right)}{\left(-4 d_{1} \lambda_n +f^{\prime}(1)\right)\left(-4 d_{1} \lambda_n\tau -1\right) h_{2}^2- 4d_{2} \lambda_n h_{1}h_2}, \\
	    					&\Theta_{3}^{2}=\frac{\left(-4 d_{1} \lambda_n\tau -1\right)\left[h_{1}h_2\left(\frac{f^{\prime \prime}(1)}{2} - 2 d_{2} \alpha_n \lambda_{n}\right)-\left(-4 d_{1} \lambda_n +f^{\prime}(1)\right)\beta h_1\right]}{\left(-4 d_{1} \lambda_n +f^{\prime}(1)\right)\left(-4 d_{1} \lambda_n\tau -1\right)h_{2}^2- 4d_{2} \lambda_n h_{1}h_2}.
	    			\end{split}
	    		\end{equation}
	    		
 If $d_{2,N}^{\prime \prime}(0)<0$, the bifurcation at $d_2=d_{2,N}^S$ is supercritiacl and the bifurcation steady states are locally asymptotically stable; if $d_{2,N}^{\prime \prime}(0)>0$, the bifurcation at $d_2=d_{2,N}^S$ is subcritical and the bifurcation steady states are unstable; all other bifurcating branches with $n\ne N$ are unstable, where  $d_{2,N}^S$ is defined as in \eqref{eq:3.12}.	      		
	    	\end{romanenum}
	    \end{theorem}
	    \begin{proof}
	    To apply the bifurcation theorems in \cite{Crandall1970Bifurcation},  fixing $d_1,\tau > 0$, we define a nonlinear mapping $F : \mathbb{R^+}  \times X^2  \times Y \to Y^3$ by
	    \begin{equation*}\label{eq:3.18}
	    F(d_2, U) =
	    \begin{pmatrix}
	    	d_{1} \Delta u + d_2 div (u\nabla v) + f(u) \\
	    	d_{1} \Delta v + \frac{1}{\tau}(a - v) \\
	    	h(u) - (\mu + \beta u)a
	    \end{pmatrix},
	  \end{equation*}
	    where $U = (u, v, a)$. Let $U_*=(1, \theta, \theta)$. It is clear that $F(d_2, U_*) = 0$ for any $d_2 > 0$, and the Fr\'{e}chet derivative of $F$ at $(d_2, U_*)$ with respect to $(u, v, a)$ is
	   \begin{equation*}\label{eq:3.19}
	    F_U\left( d_{2,n}^S, U_*\right)
	    \begin{pmatrix}
	    	\phi \\
	    	\psi \\
	    	\varphi
	    \end{pmatrix}
	    =
	    \begin{pmatrix}
	    	d_{1} \Delta \phi + d_{2,n}^S \Delta \psi + f'(1) \phi \\
	    	d_{1} \Delta \psi + \frac{1}{\tau}(\varphi - \psi) \\
	    	h_1 \phi + h_2 \varphi
	    \end{pmatrix}
	    := L
	    \begin{pmatrix}
	    	\phi \\
	    	\psi \\
	    	\varphi
	    \end{pmatrix}.
	  \end{equation*}
	
	   (1) From the definition of $d_{2,n}^S$ in  \eqref{eq:3.10}, we have $C_n\left( d_{2,n}^S\right)  = 0$. Then $\sigma=0$ is an eigenvalue of \eqref{eq:3.2} with a corresponding eigenvector $q = (\phi, \psi, \varphi)^T = (1, \alpha_n, l_n)^T \phi_n$ with $\alpha_n = -\frac{h_1}{h_2\left(d_1\lambda_n\tau+1\right)}$, $l_n = -\frac{h_1}{h_2}$. Since $\lambda_n$ is a simple eigenvalue and $d_{2,n}^S \ne d_{2,k}^S$ for  $k \ne n$,  $\sigma = 0$ is a simple eigenvalue of $L$ and the null space is
	    $N(L) = \text{span} \left\{  (1, \alpha_n, l_n)^T \phi_n \right\}$.  Thus, $\dim(N(L)) = 1.$

	    (2) Let $q^* \in N(L^*)$, where $L^*$ is the adjoint operator of $L$ defined as
	    \begin{equation*}\label{eq:3.20}
	    	L^*
	    	\begin{pmatrix}
	    		\phi \\
	    		\psi \\
	    		\varphi
	    	\end{pmatrix}
	    	=
	    	\begin{pmatrix}
	    		d_1 \Delta \phi +  f'(1) \phi + h_1 \varphi \\
	    		d_1 \Delta \psi + d_{2,n}^S  \Delta \phi - \frac{1}{\tau}\psi \\
	    		\frac{1}{\tau
	    		} \psi + h_2 \varphi
	    	\end{pmatrix}.
	    \end{equation*}
Then $N(L^*) = \mathrm{span} \left\{ (1, r_n, s_n)^T \phi_n \right\}$, where $r_n = -\frac{\tau h_2\left(d_1\lambda_n-f'(1) \right)}{h_1} ,
	    s_n = \frac{\left( d_1 \lambda_ n - f'(1) \right)}{h_1} $,
	    and $R(L)$ can be represented as \[ R(L)=\left\{\left( f_1, f_2, f_3 \right)  \in Y^3:\int_{\Omega} \left( f_1 + r_n f_2 + s_n f_3\right)  \phi_n dx = 0\right\}. \]
	    Therefore, $\mathrm{codim}(R(L)) = 1 = \dim(N(L^*))$.
	
	   (3) Now we need to verify that $ F_{d_{2U}}\left( d_{2,n}^S,  U_*\right) [q]  \notin R(L)$.
      From
	   \begin{equation*}\label{eq:3.21}
	    F_{d_{2U}}\left( d_{2,n}^S, U_*\right) [q] =
	    \begin{pmatrix}
	    	\Delta\psi \\0\\ 0
	    \end{pmatrix}
	    =
	    \begin{pmatrix}
	    	-\lambda_n\psi\\0 \\ 0
	    \end{pmatrix}
	     =
	    \begin{pmatrix}
	    	-\lambda_nh_n\phi_n\\0 \\ 0
	    \end{pmatrix},
	   \end{equation*}
	   we have
	    \[
	    \int_{\Omega}(-\lambda_n \alpha_n \phi_n+ 0 + 0 )\phi_n  dx = -\lambda_n \alpha_n \int_{\Omega} \phi_n^2  dx \ < 0, \ \  n\in \mathbb{N}.
	    \]
	    Therefore, $ F_{d_{2U}}\left( d_{2,n}^S, U_*\right) [q]  \notin R(L)$.
	
	    According to Crandall-Rabinowitz's bifurcation theorem \cite{Crandall1970Bifurcation}, we obtain the existence of non-constant steady states of (\ref{eq:2.2}). This completes the proof of part (i).
	
	    (4) Now we consider the bifurcation direction and the stability of the branching solution on  $\Gamma_n$ on the domain $\Omega = (0, l\pi)$. On the domain $\Omega = (0, l\pi)$, we have $\lambda_n = \frac{n^2}{l^2}$ and $\phi_n = \cos\left(\frac{n x}{l}\right), n\in \mathbb{N}$. Then $ q = (1, \alpha_n,l_n)^T \cos\left(\frac{n x}{l}\right)$.
	
	   In the vicinity of a bifurcation point $d_2=d_{2,n}^S$, the  bifurcation direction can be characterized by the derivatives  $d_{2,n}^{'}(0)$ and $d_{2,n}^{''}(0)$, as established in
\cite{Shi1999}. If $d_{2,n}^{'}(0)\neq 0$,  a transcritical bifurcation occurs, with
inhomogeneous solutions exist on both sides of $d_{2,n}^S$ in a neighborhood of the bifurcation point. If $d_{2,n}^{'}(0)=0$ and $d_{2,n}^{''}(0)\neq 0$, a pitchfork bifurcation occurs. More precisely, the bifurcation is supercritical (forward) when $d_{2,n}^{''}(0)> 0$, yielding two inhomogeneous solution branches for $d_2>d_{2,n}^S$ and none for $d_2<d_{2,n}^S$; the bifurcation is subcritical (backward) when $d_{2,n}^{''}(0)<0$, exhibiting two inhomogeneous solution branches for $d_2<d_{2,n}^S$ and none for $d_2>d_{2,n}^S$.

 From \cite{Shi1999},
	    \begin{equation*}\label{eq:3.22}
	    d_{2,n}'(0) = - \frac{\langle k, F_{UU}\left( d_{2,n}^S, 1, \theta,\theta\right) [q,q] \rangle}{ 2 \langle k, F_{d_{2}U}\left( d_{2,n}^S, 1, \theta,\theta\right) [q] \rangle}.
	    \end{equation*}
Direct calculation shows that
	    \begin{equation*}\label{eq:3.23}
	    F_{UU}\left( d_{2,n}^S, 1, \theta,\theta\right) [q, q] =
	    \begin{pmatrix}
	    	\frac{f''(1)}{2}+\left(\frac{f''(1)}{2} - 2 d_{2,n}^S \alpha_n \lambda_n\right) \cos \left(\frac{2nx}{l}\right)  \\
	    	0 \\
	    	-2\beta l_n \cos^2 \left(\frac{nx}{l}\right).
	    \end{pmatrix}.
	    \end{equation*}
	    Then
	   \[  \begin{aligned}
	    &\langle k, F_{UU}\left( d_{2,n}^S, 1, \theta,\theta\right) [q, q] \rangle \\
	    =&\int_0^{l\pi} \left[\left( \frac{f''(1)}{2}+\left( \frac{f''(1)}{2} - 2 d_{2,n}^S \alpha_n \lambda_n\right) \cos (\frac{2nx}{l})\right)-2\beta l_n \cos^2 \left(\frac{nx}{l}\right)\right]  \cos\left(\frac{nx}{l}\right) dx\\
	    =&0,
	\end{aligned} \]
 where $k \in (Y^3)^*$ satisfies $N(k) = R(L)$. Thus $ d_{2,n}'(0)=0$.

	    We further  to investigate the  direction of the pitchfork bifurcation which is determined by  $d_{2,n}^{\prime \prime}(0)$.
	    From \cite{Shi1999}, $d_{2,n}^{\prime \prime}(0)$ is given by
	   \begin{equation}\label{eq:3.24}
	    d_{2,n}^{\prime \prime}(0)=
	    - \frac{\langle k, F_{UUU}\left( d_{2,n}^S, 1, \theta,\theta\right) [q, q, q] \rangle + 3 \langle k, F_{UU}\left( d_{2,n}^S, 1, \theta,\theta \right) [q, \Theta] \rangle}
	    {3 \langle k, F_{d_2U}\left( d_{2,n}^S, 1, \theta,\theta\right) [q] \rangle},
	    \end{equation}
	    where $\Theta = (\Theta_1, \Theta_2, \Theta_3)$ is the unique solution of
	    \begin{equation}\label{eq:3.25}
	    F_{UU}\left( d_{2,n}^S, 1, \theta,\theta\right)  [q,q] + F_{U}\left( d_{2,n}^S, 1, \theta,\theta\right) [\Theta] = 0.
    	\end{equation}
	    Direct calculation shows that
	    \begin{equation*}
	    F_{UUU}(d_{2,n}^S, 1, \theta,\theta)
	    \begin{pmatrix}
	    	\phi \\
	    	\psi \\
	    	\varphi
	    \end{pmatrix}^3
	    =
	    \begin{pmatrix}
	    	f^{\prime \prime \prime}(1) \phi^3 \\
	    	0 \\
	    	0
	    \end{pmatrix}.
    	\end{equation*}
	    Then
	   \begin{equation}\label{eq:3.26}
\begin{aligned}
	    &\langle k,  F_{UUU}\left( d_{2,n}^S, 1, \theta,\theta\right) [q, q, q] \rangle
	    = \int_0^{l\pi} f^{\prime \prime \prime}(1) \cos^4 \left( \frac{n x}{l} \right) dx
	    = \frac{3l \pi}{8}f^{\prime \prime \prime}(1),\\
	   & \langle k, F_{d_{2}U}\left( d_{2,n}^S, 1, \theta,\theta\right) [q] \rangle = -\lambda_n \alpha_n \int_{\Omega} \phi_n^2 \ dx
	    = -\lambda_n \alpha_n \int_{0}^{l\pi} {\cos^2 \left(\frac{nx}{l}\right)} dx= - \frac{\lambda_n \alpha_n l\pi}{2}.
\end{aligned}
	   \end{equation}
 Assume that $\Theta = (\Theta_1, \Theta_2, \Theta_3)$ has the following form
	    \begin{equation}\label{eq:3.27}
	    \Theta_1 = \Theta_1^1 + \Theta_1^2 \cos\left( \frac{2  n x}{l} \right),
	    \Theta_2 = \Theta_2^1 + \Theta_2^2 \cos\left( \frac{2  n x}{l} \right),
	    \Theta_3 = \Theta_3^1 + \Theta_3^2 \cos\left( \frac{2 n x}{l} \right).
	    \end{equation}
	    Substituting (\ref{eq:3.27}) into $F_U$, we have
	    \begin{equation}\label{eq:3.28}
	    	\begin{aligned}
	    F_U\left( d_{2,n}^S, 1, \theta,\theta\right)  [\Theta] &=
	    \begin{pmatrix}
	    	d_1 \Delta \Theta_1 + d_{2,n}^S\Delta \Theta_2 + f'(1) \Theta_1 \\
	    	d_1 \Delta \Theta_2 + \frac{1}{\tau} (\Theta_3- \Theta_2) \\
	    	h_1 \Theta_1 + h_2 \Theta_3
	     \end{pmatrix}\\
	     	&=
	     \begin{pmatrix}
	     	f^\prime (1) \Theta_1^1 +\left(-4 d_1 \lambda_n \Theta_1^2 - 4 d_{2,n}^S \lambda_n \Theta_2^2 + f^\prime (1)\Theta_1^2\right)  \cos\left(\frac{2 n x}{l}\right) \\
	     	-4 d_1 \lambda_n  \Theta_2^2 \cos\left( \frac{2 n  x}{l} \right) + \frac{1}{\tau}\left[\Theta_3^1 - \Theta_2^1+(\Theta_3^2 - \Theta_2^2) \cos\left(\frac{2 n x}{l}\right)\right]  \\
	     	h_1 (\Theta_1^1 + \Theta_1^2 \cos(\frac{2 n x}{l})) + h_2 \left(\Theta_3^1 + \Theta_3^2 \cos\left(\frac{2 n x}{l}\right)\right)
	     \end{pmatrix}.
	\end{aligned}
	\end{equation}
	From (\ref{eq:3.23}) and (\ref{eq:3.24}), we have
	\begin{equation}\label{eq:3.29}
		\begin{aligned}
			F_U\left( d_{2,n}^S, 1, \theta,\theta\right)  [\Theta] &=- F_{UU}\left( d_{2,n}^S, 1, \theta,\theta\right) [q, q] \\
			&= -
			\begin{pmatrix}
				\frac{f^{\prime \prime}(1)}{2} + \left(\frac{f^{\prime \prime}(1)}{2} - 2 d_{2,n}^S \alpha_n \lambda_n\right) \cos^2\left( \frac{2n  x}{l} \right) \\
				0 \\
				-\beta l_n \left( \cos\left( \frac{2n x}{l} \right) + 1 \right)
			\end{pmatrix}.
		\end{aligned}
	\end{equation}
	 Comparing the coefficients of (\ref{eq:3.28}) and (\ref{eq:3.29}) and applying Cramer's rule, we obtain $\Theta_{1}^{1}, \Theta_{1}^{2}, \Theta_{2}^{1}, \Theta_{2}^{2}, \Theta_{3}^{1}, \Theta_{3}^{2}  $ as defined in \eqref{eq:3.17}.	
	  Then
	    \[
	    F_{UU}\left( d_{2,n}^S,1,\theta,\theta\right)  [q,\Theta]=
	    \begin{pmatrix}
	    	2d_{2,n}^S \nabla \left(\phi \nabla{\Theta_2}+\Theta_1 \nabla \psi\right)+f^{\prime \prime }(1)\phi \Theta_1 \\
	    	0 \\
	    	-\beta \left(\phi \Theta_3+\varphi \Theta_1\right)
	 \end{pmatrix}
	 =
	 \begin{pmatrix}
	    F_{UU_1} \\
	    0 \\
	    F_{UU_3}
	 \end{pmatrix},
	    \]
	    where
	    \[ \begin{aligned}
	    F_{UU_1} =&
	    4d_{2,n}^S \lambda_n \left( \Theta_2^2
	    + \alpha_n \Theta_1^2\right)\sin \left( \frac{nx}{l} \right) \sin\left( \frac{2nx}{l} \right)+\left(f^{\prime \prime}(1)-2d_{2,n}^S\lambda_n \alpha_n\right)\Theta_1^1 \cos\left( \frac{n x}{l} \right)\\
	    &+ \left(f^{\prime \prime}(1)\Theta_1^2
	    - 2d_{2,n}^S \lambda_n \alpha_n \Theta_1^2
	    - 8d_{2,n}^S \Theta_2^2\lambda_n\right) \cos\left( \frac{nx}{l} \right) \cos\left( \frac{2nx}{l} \right),
	   \end{aligned} \]
	    \[
	    F_{UU_3} =
	    -\beta\left[
	    \cos\left( \frac{nx}{l} \right)\left(\Theta_3^1+\Theta_3^2\cos\left(\frac{2nx}{l}\right)\right)
	    + l_n\cos\left( \frac{nx}{l} \right)\left( \Theta_1^1+\Theta_1^2\cos\left(\frac{2nx}{l}\right)\right)
	    \right],
	    \]
	    and
\begin{equation}\label{kuu}
	   \begin{aligned}
	    	&\langle k,F_{UU}\left( d_{2,n}^S,1,\theta,\theta\right)  [q,\Theta] \rangle \\
	    	=&l\pi d_{2,n}^S \lambda_n\left(\Theta_2^2+\alpha_n\Theta_1^2\right)+\frac{l\pi}{2}\left[\left(f^{\prime \prime}(1)-2d_{2,n}^S\lambda_n\alpha_n\right)\Theta_1^1-\beta l_n\Theta_1^3-\beta {l_n}^2\Theta_1^1\right]\\
	    	&+\frac{l\pi}{4}\left( f^{\prime \prime}(1)\Theta_1^2-2d_{2,n}^S\lambda_n\alpha_n\Theta_1^2-8d_{2,n}^S\Theta_2^2\lambda_n-\beta l_n\Theta_3^2-{l_n}^2\Theta_1^2\right) .
	    \end{aligned}
	\end{equation}
	    Substituting \eqref{eq:3.26} and \eqref{kuu} into \eqref{eq:3.24}, we obtain
	    \[
	    d_{2, n}^{\prime \prime}(0) = \frac{ m_1}{4 \lambda_n \alpha_n}, n\in \mathbb{N},
	    \]
	    where $m_1$, $\alpha_n$ defined in (\ref{eq:3.16}).
	
	    (5) By applying  Theorem 5.4 in \cite{LIU2018}, there exist $C^2$ functions $m(d_2)$: $(d_{2,n}^S-\varepsilon,\, d_{2,n}^S+\varepsilon) \to \mathbb{R}$ and $\sigma(s): (-\delta, \delta) \to \mathbb{R}$, such that
	   \begin{equation}\label{muu}
\begin{aligned}
	    	F_U\left( d_2, 1, \theta, \theta\right)  [\phi(d_2),\, \psi\left( d_2\right) ,\, \varphi(d_2)] =& m(d_2) \, K[\phi(d_2),\, \psi(d_2),\, \varphi(d_2)],\\
	    	&\text{for} \, d_2 \in \left( d_{2,n}^S-\varepsilon,\, d_{2,n}^S+\varepsilon\right) ,
	    \end{aligned}
	 \end{equation}
	    and
	    \[\begin{aligned}
	    	F_U(d_{2,n}(s), U_n(s), V_n(s),A_n(s))[\Lambda(s),\, \Phi(s),\, R(s)]  = &\sigma(s) \, K[\Lambda(s),\, \Phi(s),\, R(s)],\\
	    	&\text{for}\, s \in [-\delta, \delta],
	    \end{aligned}
	    \]
	    with
	    \[
	    m\left( d_{2,n}^S\right)  = \sigma(0) = 0, \quad
	    \left( \phi\left( d_{2,n}^S\right) , \psi\left( d_{2,n}^S\right) , \varphi\left( d_{2,n}^S\right) \right)
	    = (1, h_n, l_n)^T \cos\left( \frac{n x}{l} \right).
	    \]
	    Moreover, near $s=0$, the functions $\sigma(s)$ and $-s d_{2,n}^{\prime}\left( s\right)  m'\left( d_{2,n}^S\right) $ have  same zeroes, and whenever $\sigma(s) \neq 0$ they have the same sign, and satisfy
	    \begin{equation}\label{eq:3.30}
	    	\lim_{s \to 0} \frac{-s d_{2,n}^{\prime}(s) m'\left( d_{2,n}^S\right) }{\sigma(s)} = 1.
	    \end{equation}
	    Here $K: X^2\times Y \to Y^3$ is the inclusion map $K(u) = u$. The stability of bifurcating nonconstant steady states can be determined by the sign of $\sigma(s)$.
	
	     From \eqref{muu}, we have
 \begin{equation}\label{mda}
m^3+A_n(d_2)m^2+B_n(d_2)m+C_n(d_2)=0,
\end{equation}
where $A_n(d_2), B_n(d_2), C_n(d_2)$ are defined as in \eqref{eq:3.4}. Differentiating \eqref{mda} with respect to $d_2$ and noting that $m(d_{2,n}^S)=0$, we have
$m'(d_{2,n}^S)=-\frac{h_1\lambda_n }{\tau B_n}<0, n\in \mathbb{N}.$
 Then at $d_2 = d_{2,N}^S$, we have
	     \begin{equation}\label{md2}
	    m'\left( d_{2,N}^S\right)  = -\frac{h_1\lambda_N}{\tau B_N} < 0.
	     \end{equation}	
	    If $d_{2,N}^{\prime \prime}(0)< 0$, then for sufficiently small $\delta>0$, $d_{2,N}^{\prime \prime}(s)< 0$ when $s \in (-\delta, \delta)$, which results in  $d_{2,N}^\prime(s)> 0$ for $s \in (-\delta, 0)$, and
	    $d_{2,N}^\prime(s)< 0$ for $s \in (0, \delta)$. This along with \eqref{md2} leads to $-s d_{2,N}^{\prime}(s) m'\left( d_{2,N}^S\right) < 0$ for $s \in (-\delta, \delta)\setminus \{ 0 \}$.
	    Due to (\ref{eq:3.30}), we have $\sigma(s) < 0$ for $s \in (-\delta, \delta) \setminus \{ 0 \}$.
	    Therefore, the bifurcating solutions are locally asymptotically stable if $d_{2,N}^{\prime \prime}(0)< 0$.
	    Similarly, when $d_{2,N}^{\prime \prime}(0)> 0$,
we have $\sigma(s) > 0$ for $s \in (-\delta, \delta) \setminus \{ 0 \}$, which implies that the bifurcating solutions are unstable if $d_{2,N}^{\prime \prime}(0)> 0$.
	For any other bifurcation at $d_2 = d_{2,n}^S \neq d_{2,N}^S$, the steady state $(1, \theta, \theta)$ is unstable at the bifurcation point, and all bifurcating solutions are unstable.
	    \end{proof}
	
	    \subsection{Strong kernel case}

In this subsection, we will study the stability of the positive constant steady state $\left( 1, \theta\right)  $ and the associated bifurcations of (\ref{eq:1.4}) with strong temporal kernel by choosing $d_2$ as the bifurcation.

	    From Lemma 2.3, for the strong kernel $ g_{1}(t) = \frac{t}{\tau^2} e^{-\frac{t}{\tau}}$, the stability of the steady state $\left( 1, \theta\right)  $ of the original system (\ref{eq:1.4}) can be equivalently analyzed through the equivalent system \eqref{eq:2.11}.
	  System (\ref{eq:2.11}) admits a constant steady state $(1, \theta,\theta,\theta)$. Linearizing system (\ref{eq:2.11}) at $(1, \theta,\theta,\theta)$ yields the following eigenvalue problem
	    \begin{equation}\label{eq:3.32}
	    \begin{cases}
	    	d_1\Delta \phi + d_2 \Delta \psi + f^\prime(1) \phi = \sigma\phi, & x \in \Omega, \\
	    	d_1 \Delta \psi + \frac{1}{\tau}(\varphi - \psi) = \sigma\psi, & x\in\Omega, \\
	    	d_1 \Delta \varphi + \frac{1}{\tau}(\varrho - \varphi) = \sigma\varphi, & x \in \Omega, \\
	    	h_1 \phi + h_2 \varrho = \sigma \varrho, & x \in \Omega, \\
	    	\frac{\partial \phi}{\partial \vec{n}}=\frac{\partial \psi}{\partial \vec{n}}=\frac{\partial \varphi}{\partial \vec{n}}=0,  & x \in \partial \Omega,
	    \end{cases}
	    \end{equation}
	    where  $h_1, h_2$ are defined in \eqref{h}.
	    The eigenvalues of (\ref{eq:3.32}) coincide with those of the Jacobian matrix:
	    \[
	    J_{n}^s =
	    \begin{pmatrix}
	    	- d_1 \lambda_n +f^\prime(1) & - d_2 \lambda_n & 0 & 0 \\[0.4em]
	    	0 & - d_1 \lambda_n - \frac{1}{\tau} & \frac{1}{\tau} & 0 \\[0.4em]
	    	0 & 0 & - d_1 \lambda_n - \frac{1}{\tau} & \frac{1}{\tau} \\[0.4em]
	    	h_1 & 0 & 0 & h_2
	    \end{pmatrix}.
	    \]
	    Then the corresponding characteristic equation is
	    \begin{equation}\label{eq:3.33}
	    \sigma^4 + D_n(d_2) \sigma^3 + E_n(d_2)\sigma^2 + G_n(d_2)\sigma + T_n(d_2) = 0,
	    \, n \in \mathbb{N}_0,
	   \end{equation}
	    with
	    \begin{equation}\label{eq:3.34}
	    \begin{aligned}
	    	& D_n(d_2) =  3 d_1 \lambda_n + \frac{2}{\tau} - f^\prime(1) - h_2,\\
	    	&E_n(d_2) = \left( d_1 \lambda_n + \frac{1}{\tau}\right) ^2+ 2\left( d_1 \lambda_n - f^{\prime}(1) - h_2\right) \left( d_1 \lambda_n + \frac{1}{\tau}\right)  - h_2\left( d_1 \lambda_n - f^{\prime}(1)\right) , \\[0.4em]
	    	&G_n(d_2)= \left( d_1 \lambda_n - f^{\prime}(1) - h_2\right) \left( d_1 \lambda_n + \frac{1}{\tau}\right) ^2 - 2h_2\left( d_1 \lambda_n - f^{\prime}(1)\right) \left( d_1 \lambda_n + \frac{1}{\tau}\right) , \\[0.4em]
	    	&T_n(d_2) = -h_2d_1^3 \lambda_n^3 + \left( f^{\prime}(1)h_2 -\frac{2h_2}{\tau}\right) d_1^2 \lambda_n^2
	    	+ \left(\frac{d_2h_1 - d_1h_2}{\tau^2}+ \frac{2f^{\prime}(1)d_1h_2}{\tau} \right)\lambda_n + \frac{f^{\prime}(1)h_2}{\tau^2}.
	    \end{aligned}
	    \end{equation}

By the Routh-Hurwitz stability criterion,  the matrix $J_{n}^s$ is stable if and only if:
	    \begin{equation*}\label{eq:3.35}
\begin{aligned}
	   & D_n(d_2) > 0,  T_n(d_2) > 0, R_n(d_2):=D_n(d_2) E_n(d_2) - G_n(d_2) > 0, \\
& S_n(d_2):=D_n(d_2) E_n(d_2) G_n(d_2) - D_n^2(d_2) T_n(d_2) - G_n^2(d_2) > 0,
\end{aligned}
	    \end{equation*}
for all $ n\in \mathbb{N}_0$.
Then the matrix $J_n^s$ may lose its stability either when there exists  $d_2$  such that
	$T_n(d_2) = 0$ for some $n\in \mathbb{N}_0$, in which case  $J_n^s$ exhibits a zero eigenvalue, or when there exists  $d_2$  such that  $S_n(d_2)= 0$, $D_n(d_2) > 0$ and $R_n(d_2)>0$ for some $n\in \mathbb{N}_0$, corresponding to a pair of purely imaginary eigenvalues $\pm \tilde{\omega}_0i$, where $\tilde{\omega}_0^2=\frac{G_n(d_2)}{D_n(d_2)}$. Note that for $n=0 (\lambda_0=0)$, the four eigenvalues of $J_0^s$ are $f'(1), -\frac{1}{\tau}, -\frac{1}{\tau}$ and $h_2$ respectively, which are all negative. Then we only need to consider 	$ n\in \mathbb{N}$.
For $ n\in \mathbb{N}$,	it follows from  (H1)  and the  definition of $h_2$ that $D_n(d_2) > 0$, $E_n(d_2)>0$, $G_n(d_2)>0$ and $R_n(d_2)>0$. Then we shall investigate the existence of $d_2$ such that $T_n(d_2) = 0$ or $S_n(d_2)= 0$ for some $n\in \mathbb{N}$.

	   For $d_2\in \mathbb{R}, p\in (0,+\infty)$, define the following functions
	   	\begin{equation}\label{eq:3.36}
	   		\begin{aligned}
	   			D(d_2,p) :=& 3 d_1 p + \frac{2}{\tau} - f^{\prime}(1) - h_2, \\
	   			E(d_2,p)  := &\left(  d_1 p + \frac{1}{\tau}\right) ^2
	   			+ 2\left(  d_1 p - f^{\prime}(1) - h_2 \right)  \left( d_1 p + \frac{1}{\tau} \right)
	   			- h_2 \left(  d_1 p - f^{\prime}(1) \right) , \\
	   			G(d_2,p)  :=& \left(  d_1 p - f^{\prime}(1) - h_2 \right) \left( d_1 p + \frac{1}{\tau} \right)^2
	   			-2 h_2\left(  d_1 p - f^{\prime}(1)\right) \left( d_1 p + \frac{1}{\tau} \right), \\
	   			T(d_2,p)  :=& -h_2 d_1^3 p^3
	   			+ \left( f^{\prime}(1) h_2 -  \frac{2h_2}{\tau}\right) d_1^2 p^2
	   			+ \left(\frac{d_2 h_1 - d_1 h_2}{\tau^2} + \frac{2 f^{\prime}(1) h_2 d_1}{\tau}\right)p\\
	   			&+ \frac{f^{\prime}(1) h_2 }{\tau^2},\\
	   			R(d_2,p) :=& D(d_2,p) E(d_2,p) - G(d_2,p)=8 d_1^3 p^3 + a_3 p^2 + b_3 p + c_3,\\
	   			S(d_2,p) :=& D(d_2,p) E(d_2,p) G(d_2,p) - D^2(d_2,p) T(d_2,p) - G(d_2,p) \\
	   			=& 8d_1^6 p^6 + a_4 p^5 + b_4 p^4 + c_4 p^3 + d_4 p^2 + e_4 p + f_4,
	   		\end{aligned}
	   	\end{equation}
	  where $a_i, b_i, c_i,d_4,e_4, f_4  (i=3,4)$ are given in Appendix C.
	  Solving  $T(d_2,p) = 0$ for $d_2$ gives
	  \begin{equation}\label{eq:3.38}
	  \tilde{d}_2^S(p) = \frac{h_2 \left( - f^\prime(1) + d_1 p \right) (d_1 p \tau + 1)^2}{h_1 p}.
	  \end{equation}
 Solving $S(d_2,p)=0$ for $d_2$ yields
	  \begin{equation}\label{eq:3.39}
	  	\begin{aligned}
	  \tilde{d}_2^H(p) =& -\frac{1}{h_1 p^3 \tau^2 } (1 + d_1 p \tau) \Big[ d_1p^2\left(1+d_1p\tau\right)^2(-1+d_1+(-3+d_1)d_1p\tau)+h_2^2(1+3d_1\\
	  &p\tau)(1+p\tau(-2+4d_1+3(-1+d_1)d_1p\tau))+(f^\prime(1))^2(-1+2h_2\tau-d_1p\tau)(-((1\\
	  &+(-2+d_1)p\tau)(1+d_1p\tau))+h_2\tau(2+(-1+2d_1)p\tau))-h_2p(1+d_1p\tau)(-1+\\
	  &d_1(2+p\tau(-8+8d_1+p\tau+6(-2+d_1)d_1p\tau)))+f^\prime(1)(-p(1+d_1p\tau)^2(-1+2d_1\\
	  &+d_1(-5+2d_1)p\tau)+h_2^2\tau(-4+p\tau(5-16d_1+3(3-4d_1)d_1p\tau))+h_2(1+d_1p\tau)\\
	  &(2+p\tau(-6+p\tau+2d_1(6+(-8+5d_1)p\tau))))\Big].
	  \end{aligned}
	  \end{equation}
 From the hypotheses (H1)-(H3) and the definitions of $h_1$, $h_2$ in \eqref{h}, we have  $S(d_2,p) >0$ for $d_2 < 0, p>0$. This implies that $\tilde{d}_2^H(p)>0$.
	  Define
	  \begin{equation}\label{eq:3.40}
	  	\tilde{d}_{2,n}^S:= \tilde{d}_2^S(\lambda_n),\ \ \  \ \tilde{d}_{2,n}^H := \tilde{d}_2^H(\lambda_n),\,  n \in \mathbb{N},
	  \end{equation}
	  where $\tilde{d}_2^S(p)$ and $\tilde{d}_2^H(p)$ are defined as in (\ref{eq:3.38}) and (\ref{eq:3.39}).
Obviously, $T_n(\tilde d_{2,n}^S) = 0$ and $S_n(\tilde d_{2,n}^H)=0$. This implies that $\sigma=0$ is a root of Eq.(\ref{eq:3.33}) at $d_2=\tilde d_{2,n}^S$ and $\sigma=\pm i\tilde{\omega} _0$ is a pair of purely imaginary roots of Eq.(\ref{eq:3.33}) at $d_2=\tilde d_{2,n}^H$.
In the following, we will show that system \eqref{eq:2.11} admits Hopf bifurcations at $d_2=\tilde d_{2,n}^H$ and steady state bifurcations at $d_2=\tilde d_{2,n}^S$.

 First we demonstrate the transversality condition for the occurrence of purely imaginary eigenvalues at $d_2 = \tilde{d}_{2,n}^H$.
	  \begin{lemma}\label{lem3.6}
	  	Let $\tilde{d}_{2,n}^H$ be defined in $(\ref{eq:3.40})$. For $d_2$ near $\tilde{d}_{2,n}^H$, Eq.$(\ref{eq:3.33})$ has a pair of purely imaginary eigenvalue $\sigma = \alpha(d_2) \pm i \omega(d_2)$ with $\alpha(\tilde{d}_{2,n}^H) = 0$ and $\alpha^\prime(\tilde{d}_{2,n}^H)> 0.$
	  \end{lemma}
	  \begin{proof}
	  	Obviously, $\alpha(\tilde{d}_{2,n}^H)=0$. Differentiating both sides of the Eq.(\ref{eq:3.33}) with respect to $d_2$,
and noting that
 ${\tilde{\omega}_0}^2=\frac{G_n(\tilde{d}_{2,n}^H)}{D_n(\tilde{d}_{2,n}^H)}$ at $d_2 = \tilde{d}_{2,n}^H$, we have
	  	\begin{align*}
	  		\left. \frac{d\sigma}{d(d_2)} \right|_{d_2 = \tilde{d}_{2,n}^H}
	  		&= \frac{h_1 \lambda_n}{\tau^2}\cdot \frac{ 2\tilde{\omega}_0^2 D_n(\tilde{d}_{2,n}^H) + i\tilde{\omega}_0 (2E_n(\tilde{d}_{2,n}^H) - 4\tilde{\omega}_0^2) }{ (2\tilde{\omega}_0^2 D_n(\tilde{d}_{2,n}^H))^2 + \tilde{\omega}_0^2 (2 E_n(\tilde{d}_{2,n}^H) - 4\tilde{\omega}_0^2)^2 }.
	  	\end{align*}
	  	Therefore,
	  	\[
	  	\alpha'(\tilde{d}_{2,n}^H)
	  	= \mathrm{Re} \left[ \left. \frac{d\sigma}{d(d_2)} \right|_{d_2 = \tilde{d}_{2,n}^H} \right]
	  	= \frac{2 h_1\lambda_n D_n(\tilde{d}_{2,n}^H)}{\tau^2 \left[ (2\tilde{\omega}_0 D_n(\tilde{d}_{2,n}^H))^2 + (2E_n(\tilde{d}_{2,n}^H) - 4\tilde{\omega}_0^2)^2 \right]} > 0.
	  	\]
	  \end{proof}
	
	  According to Amann \cite{A1991hopffenzhi}, the Hopf bifurcation theorem for the strong kernel system reads as follows.
	  \begin{theorem}\label{3.9}
	  Let $\tilde d_{2,n}^H$ and $D_n$, $G_n$ be defined as in $(\ref{eq:3.40})$ and $(\ref{eq:3.34})$. Suppose that $\lambda_n$ is a simple eigenvalue of $(\ref{eq:1.6})$, and $\tilde d_{2,n}^H\ne \tilde d_{2,k}^H$ for any $ k\in \mathbb{N}$ and $k\ne n$. Then a Hopf bifurcation occurs at $d_2=\tilde d_{2,n}^H$ for  system \eqref{eq:2.11}, and there exists a family of periodic orbits of the following form:
	  	\[ \left\{\left( V_n(x,t,s),T_n(s),\tilde d_2^{(n)}(s)\right) :s\in (0,\delta) \right\}.\]
	  	Here $V_n(x,t,s)=(u_n(x,t,s),v_n(x,t,s),w_n(x,t,s),a_n(x,t,s))$ is a time-periodic solution of  system \eqref{eq:2.11} with period $T_n(s)$ when $d_2=\tilde d_2^{(n)}(s)$, satisfying
	  	\[\tilde d_2^{(n)}(0)=\tilde d_{2,n}^H ,\,\lim_{s \to 0} V_n(x,t,s)=(1,\theta,\theta,\theta), \, \lim_{s \to 0}T_n(s) =2\pi\sqrt{\frac{D_n(\tilde{d}_{2,n}^H)}{G_n(\tilde{d}_{2,n}^H)}}.\]
	  \end{theorem}	
	
Next, we prove that system undergoes a steady-state bifurcation and show the existence of
nonconstant steady states.
	  \begin{lemma}\label{le3.6}
	      Let $\tilde d_2^S(p)$ be defined as in (\ref{eq:3.38}). Under $(H1)-(H3)$, There exists $\tilde p >0$ such that $\tilde d_2^S(p)$ is increasing for $p \in \left ( 0,\tilde p \right ) $ and decreasing for $p \in  (\tilde p,\infty)$. Moreover, $\displaystyle \lim_{p \to 0}\tilde  d_2^S(p)=-\infty $, $\displaystyle \lim_{p \to +\infty} \tilde d_2^S(p)=-\infty $ and $\tilde d_2^S(p)$ attains its global maximum value at $p=\tilde p$.
	  \end{lemma}
	  \begin{proof}
	    Calculate the derivative of $\tilde{d}_2^S(p)$ with respect to $p$,
	  	\[ 	\frac{d}{dp} \tilde{d}_2^S(p) = \frac{h_2}{h_1} \cdot \frac{
	  			\left(d_1p\tau+1\right)\left(2d_1^2\tau p^2-f^\prime(1)d_1\tau p+f^\prime(1)\right)}{p^2}.\]
	  	Since $h_1>0,h_2<0$, the sign of the derivative is opposite to that of the cubic factor in the numerator, which has a unique positive zero $\tilde{p}$ so that $\frac{d}{dp} \tilde d_2^S(p)>0$ for $p\in \left ( 0,\tilde p \right ) $ and $\frac{d}{dp} \tilde d_2^S(p)<0$ for $p\in (\tilde p,\infty) $. The limits are obtained directly from (\ref{eq:3.38}).
	  \end{proof}
Form Lemma \ref{le3.6}, we obtain the following theorem.	
\begin{theorem}\label{th3.6}
	      Let $\tilde{d}_{2,n}^S$ be defined as in \eqref{eq:3.40}, and  $N$ be the index for which $\tilde{d}_{2,N}^S= \displaystyle \max _{n\in \mathbb{N} } \left \{ \tilde d_{2,n}^S \right \}$. Under $(H1)-(H3)$,  $\tilde{d}_{2,N}^S<0$. Moreover, $(1,\theta,\theta,\theta)$ is unstable when $d_2<\tilde{d}_{2,N}^S$, and a zero eigenvalue appears at each $d_2=\tilde d_{2,n}^S $.
	  \end{theorem}

	 Similar to the weak kernel case, we have the following steady-state bifurcation results for the strong kernel.
	  \begin{theorem}\label{th3.7}
	  	Assume $d_1$, $ d_2$, $f$ and $h$ satisfy hypotheses $(H1)-(H3)$.
	  	Suppose that $\lambda_n$ is a simple eigenvalue of $(\ref{eq:1.6})$, and $\tilde d_{2,n}^S\ne \tilde d_{2,k}^S$ for any $k\in \mathbb{N}$ and $k\ne n$. Then $d_2=\tilde d_{2,n}^S$ is a steady-state bifurcation point. Near $\left( \tilde d_{2,n}^S,1,\theta,\theta,\theta \right) $,  system $(\ref{eq:2.11}) $ has a line of homogeneous solutions $\tilde \Gamma _0:=\left\{(d_2,1,\theta,\theta,\theta);\, d_2\in \mathbb{R}\right\}$ and a smooth curve $\tilde \Gamma_n$ bifurcating from $\tilde \Gamma_0$ at $d_2=\tilde d_{2,n}^S$ in a form of
	  	\begin{equation}\label{eq:3.42}
	  		\tilde \Gamma_n=\left\{\left( \tilde d_{2,n}(s),U^n(s,x),V^n(s,x),W^n(s,x),Z^n(s,x)\right) :-\delta<s<\delta\right\},
	  	\end{equation}
	  	with
	  \begin{equation}\label{eq:3.43}
	  	 \begin{aligned}
	  		&U^n(s,x)=1+s\phi_n(x)+sg_{1,n}(s,x),\\
	  		&V^n(s,x)=\theta-\frac{sh_1\phi_n(x)}{h_2(d_1\lambda_n\tau+1)^2}+sg_{2,n}(s,x),\\
	  		&V^n(s,x)=\theta-\frac{sh_1\phi_n(x)}{h_2\left(d_1\lambda_n\tau+1\right)}+sg_{3,n}(s,x),\\
	  		&Z^n(s,x)=\theta-\frac{sh_1\phi_n(x)}{h_2}+sg_{4,n}(s,x).
	  	\end{aligned}
	  \end{equation}
	  	Where $\tilde d_{2,n}(s)$, $g_{1,n}(s,x)$,  $g_{2,n}(s,x)$, $g_{3,n}(s,x)$, $g_{4,n}(s,x)$ are smooth functions defined for $-\delta<s<\delta$ such that $\tilde d_{2,n}(0)=\tilde d_{2,n}^S$, $g_{i,n}(0,x)=0 (i=1,2,3,4)$ and $\delta$ is a positive constant.
	  	
	  	Let $\Omega=(0,l\pi)$, then  $\tilde d_{2,n}^\prime(0)=0$ and
	  	\begin{equation}\label{eq:3.44}
	  	\tilde d_{2, n}^{\prime \prime}(0) = \frac{ m_2}{4 \lambda_n j_n},\ \  n\in \mathbb{N}.
	  \end{equation}
	  	Here,
	  	\begin{equation}\label{eq:3.45}
	  		\begin{aligned}
	  			&\begin{aligned}
	  				m_2=&f^{\prime \prime \prime}(1)-8 \tilde d_{2,n}^S \lambda_n \Theta_{2}^{2} + 4 \tilde d_{2,n}^S \lambda_n j_n\Theta_{1}^{2} + 4(f^{\prime \prime}(1)-2 \tilde d_{2, n}^{S} \lambda_n j_n) \Theta_{1}^{1}+ 2f^{\prime \prime}(1) \Theta_{1}^{2}\\
	  				& +\frac{2 \beta h_1 (2 \Theta_{4}^{1}+\Theta_{4}^{2})}{h_2} -  \frac{2\beta h_1^2(2\Theta_{1}^{1}+\Theta_1^2)}{h_2^2},
	  			\end{aligned}	\\
	  			&j_n=-\frac{h_1}{h_2\left(1+d_1\lambda_n\tau\right)^2},\ \ k_n=-\frac{h_1}{h_2\left(1+d_1\lambda_n\tau\right)},
	  	\end{aligned}
	  \end{equation}
	  	and $\Theta_{1}^{1}$, $\Theta_{1}^{2}$,  $\Theta_{2}^{1}$, $\Theta_{2}^{2}$, $\Theta_{3}^{1}$, $\Theta_{3}^{2}$, $\Theta_4^1$, $\Theta_4^2$ are given by
	  	\begin{equation}
	  		\begin{aligned}\label{eq:3.46}
	  			&\Theta_1^1 = -\frac{f^{\prime \prime}(1)}{2 f^\prime(1)}, \Theta_2^1 = \Theta_3^1 = \Theta_4^1 = \frac{h_1h_2 f^{\prime \prime}(1) - 2 f^{ \prime}(1) \beta h_1}{2 h_2^2 f^{ \prime}(1)}, \\[4pt]
	  			&\Theta_1^2 = -\frac{\left(\frac{f^{\prime \prime}(1)}{2} -2 d_2 j_n\lambda_n\right) \left(4 d_1 \lambda_n\tau + 1\right)^2 h_2^2 - 4 \beta h_1 d_2 \lambda_n }{(-4d_1 \lambda_n + f^{ \prime}(1)) \left(4 d_1 \lambda_n\tau + 1\right)^2 h_2^2 + 4 h_1 h_2 d_2  \lambda_n }, \\[4pt]
	  			&\Theta_2^2 = \frac{-(-4d_1 \lambda_n+f^{ \prime}(1)) \beta h_1 + h_1 h_2\left(\frac{f^{\prime \prime}(1)}{2} - 2 d_2 j_n \lambda_n\right)}{(-4 d_1 \lambda_n+f^{\prime}(1)) \left( 4d_1 \lambda_n\tau+ 1\right) ^2 h_2^2 + 4h_1h_2d_2\lambda_n}, \\[4pt]
	  			&\Theta_3^2 = \frac{-(- 4d_1\lambda_n+f^{\prime}(1)\left(4d_1\lambda_n\tau+1\right) \beta h_1 + h_1h_2\left(4 d_1\lambda_n\tau+1\right) \left(\frac{f^{\prime \prime}(1)}{2}-2d_2j_n\lambda_n\right)}{(-4 d_1 \lambda_n+f^{\prime}(1)) \left( 4d_1 \lambda_n\tau+ \right) ^2 h_2^2 + 4h_1h_2d_2\lambda_n},\\
	  			&\Theta_4^2=\frac{\left(4d_1\lambda_n\tau+1\right)^2\left[-\left(-4d_1\lambda_n+f^{\prime}(1)\right)\beta h_1+h_1h_2\left(\frac{f^{\prime \prime}(1)}{2}-2d_2j_n\lambda_n\right)\right]}{(-4 d_1 \lambda_n+f^{\prime}(1)) \left( 4d_1 \lambda_n\tau+ 1\right) ^2 h_2^2 + 4h_1h_2d_2\lambda_n}.
	  		\end{aligned}
	  	\end{equation}
	  	If $\tilde d_{2,N}^{\prime \prime}(0)<0$, the bifurcation at $d_2=\tilde{d}_{2,N}^S$ is supercritiacl and the bifurcation steady states are locally asymptotically stable; 	if $\tilde d_{2,N}^{\prime \prime}(0)>0$,  the bifurcation at $d_2=\tilde{d}_{2,N}^S$ is subcritiacl and the bifurcation steady states are unstable; all other bifurcating steady states from $\tilde d_{2,n}^S$ with $n\ne N$ are unstable, where $\tilde{d}_{2,N}^S$ is defined as in Theorem \ref{th3.6}.
	  	
	  \end{theorem}
	
	\begin{remark}\label{3.10}
	  	For the strong kernel, all steady-state bifurcation points satisfy $\tilde{d}_{2,n}^S<0$, while Hopf bifurcation points $d_2=\tilde{d}_{2,n}^H$ can be either positive or negative. When $d_2<0$ (i.e. toward past memories) one can verify that $S_n(d_2)>0$ for all $n\in \mathbb{N}_0$, so no Hopf bifurcation occurs in that regime. 
	  \end{remark}

	\section{Numerical Simulations}

	In this section, we use a finite difference scheme to present the spatiotemporal patterns numerically on the one-dimension domain $\Omega = (0, \pi)$. For simplicity,   we choose the standard logistic growth term $f(u)=u(1-u)$ and a saturating memory-formation rate $h(u)=\frac{2\rho u^2}{1+u^2}$. Parameter values are fixed as  $d_1 = 0.1$, $\mu = \beta = 1$, $\rho = 2$, and let $d_2$ and $\tau$ be varied.

	\subsection{Weak kernel case}

	We first illustrate the predictions of Section 3.1 for the weak memory kernel. With the weak kernel $g_0(t)=\frac{ 1}{\tau}e^{-\frac{t}{\tau}}$, system (\ref{eq:1.4}) reduces to
	\begin{equation}\label{eq:4.1}
		\begin{cases}
			u_{t} = d_{1} \Delta u + d_{2}  div (u \nabla v) + u(1-u) , & x \in \Omega, \ t > 0, \\
			v_{t} = d_{1} \Delta v + \frac{1}{\tau} (a - v), & x \in \Omega, \ t > 0, \\
			a_{t} = \frac{2\rho u^2}{1+u^2} - (\mu + \beta u) a, & x \in \Omega, \ t > 0, \\
			\frac{\partial u}{\partial \vec{n}} = \frac{\partial v}{\partial \vec{n}} =0, & x \in \partial \Omega, \ t > 0.
		\end{cases}
	\end{equation}
	The system admits a positive constant steady state solution $\left( 1, \frac{\rho}{\mu + \beta}, \frac{\rho
	}{\mu + \beta} \right)$.
	According to Theorem \ref{3.4} and Theorem \ref{3.5}, the steady-state  bifurcation points are
 \begin{equation}\label{eq:4.2} d_2={d}_{2,n}^S=-\frac{2(d_1\lambda_n+1)(d_1\lambda_n\tau+1)}{\lambda_n}<0 ,\, n=1,2,...,
    \end{equation}
	and the Hopf bifurcation points are
	\begin{equation}\label{eq:4.3} d_2=d_{2,n}^H=\frac{(1+\tau+2d_1\lambda_n\tau)(1+2\tau+d_1\lambda_n\tau)(3+d_1\lambda_n)}{\tau\lambda_n}>0,\, n=1,2,...,
    \end{equation}
    where $\lambda_n=n^2$ are the eigenvalues of $-\Delta$ on $(0,\pi)$ with Neumann conditions.

  \begin{figure}[!t]
	\centering
	\includegraphics[width=0.55\textwidth]{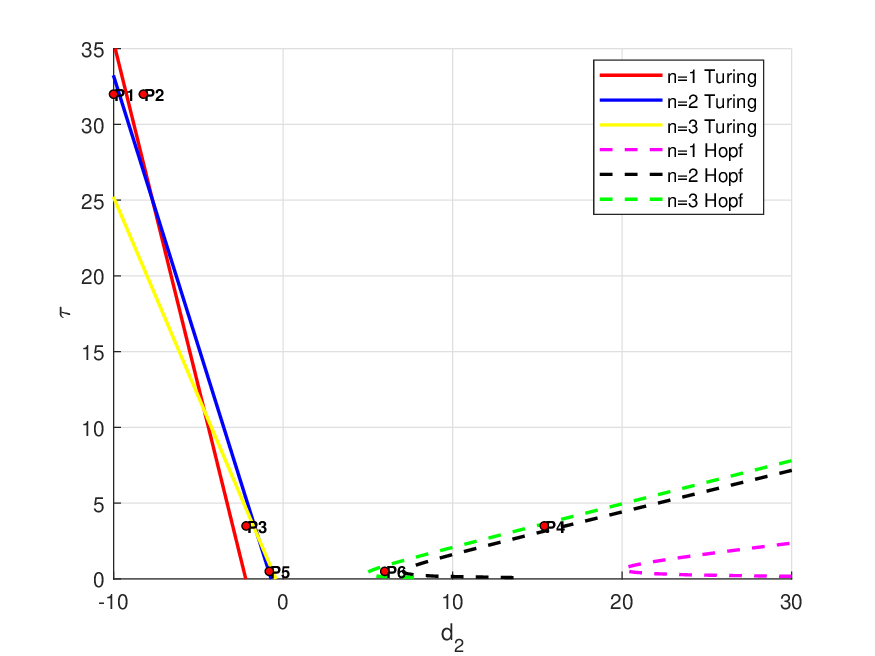}
	\caption{ The bifurcation diagram of system \eqref{eq:4.1} for the weak kernel in the $(d_2, \tau)$-plane.  The Hopf bifurcation curves $d_2=d_{2,n}^H$ defined in \eqref{eq:4.2} and the Turing bifurcation curves $d_2=d_{2,n}^S$ defined in \eqref{eq:4.3} are plotted for $n=1, 2, 3.$ Six red points $P_1 = (-10, 32),
    		P_2 = (-8.24, 32),
    		P_3 = (-2.18, 3.5),
    		P_4 = (15.41,3.5),
    		P_5 = (-0.81, 0.5),
    		P_6 = (6, 0.5)$ are chosen for simulation.
    		}
	\label{Fig.1}
\end{figure}

    \begin{figure}[htbp]
    	\centering
    	\includegraphics[width=1\linewidth,keepaspectratio]{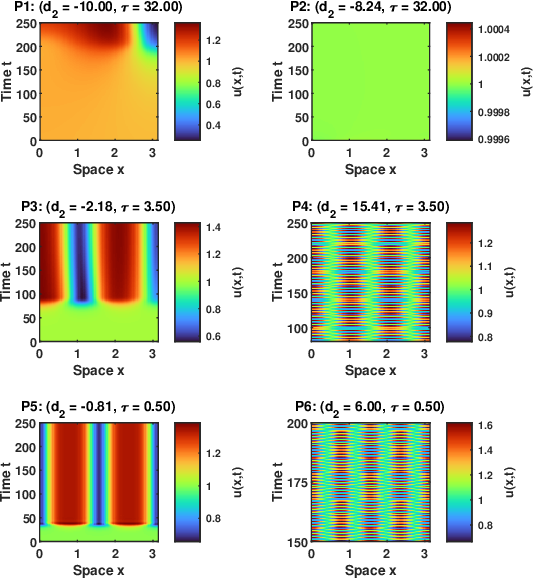}
    	\caption{This illustrates the spatiotemporal plots of the biological population density $u(x,t)$ obtained from numerical simulations under the six parameter sets described in Figure \ref{Fig.1}. The horizontal axis represents the spatial position $x$, the vertical axis represents time $t$, and the color indicates the variation in the amplitude of $u(x,t)$.}
    	\label{Fig.2}
    \end{figure}

 Figure \ref{Fig.1} shows the bifurcation diagram of system (\ref{eq:4.1}) in the $(d_2, \tau)$-plane. For a fixed $\tau$, when $d_2=d_{2,n}^H(\tau)$ varies from $d_2 = 0$ from left to right, the constant steady state $(1, 1, 1)$ loses stability at the first Hopf bifurcation curve $d_2$ defined by (\ref{eq:4.3}), thereby generating spatiotemporal patterns.
    When $d_2=d_{2,n}^S(\tau)$ varies from $d_2 = 0$ from right to left, the steady state $(1, 1, 1)$ loses stability at the first steady-state bifurcation curve $d_2$ defined in (\ref{eq:4.2}), thereby forming spatial patterns.

    When $\tau = 32.0$ (see the first line of Figure \ref{Fig.2}), the only possible spatial pattern is of the Turing type. For $d_{2} = -10 < d_{2,1}^S$ at point $P_1$, after introducing a perturbation corresponding to the most unstable wavenumber in the initial condition, the simulation results show that the system develops a spatially oscillatory structure, which can be sustained and slightly amplified. For $d_{2} = -8.24 > d_{2,1}^S$ at point $P_2$, which lies within the stable region, the homogeneous equilibrium remains stable during the time evolution.
    When $\tau = 3.50$ (see the second line of Figure \ref{Fig.2}), it can be observed that for $d_{2} = -2.18 < d_{2,2}^S$ at point $P_3$, corresponding to the $n = 2$ Turing-unstable mode, a spatial pattern with two distinct peaks emerges and remains steady over time. For $d_{2} = 15.41$ at point $P_4$, which lies in the $n = 3$ Hopf-unstable region, a structure combining pronounced spatial oscillations and temporal oscillations appears.
    When $\tau = 0.5$ (see the third line of Figure \ref{Fig.2}), it can be observed that for $d_{2} = -0.81 <d_{2,3}^S$ at point $P_5$, corresponding to the $n = 3$ Turing mode, the instability leads to the formation of a spatial pattern with three peaks, which remains stable in time. For $d_{2} = 6$ at point $P_6$, located in the $n = 3$ Hopf-unstable region, a periodic spatiotemporal oscillatory pattern is produced.

	\subsection{Strong kernel case}
	For the strong kernel case $g_1(t)=\frac{t}{\tau^2}e^{-\frac{t}{\tau}}$, system (\ref{eq:1.4}) is equivalent to
	\begin{equation}\label{eq:4.4}
		\begin{cases}
			u_{t} = d_{1} \Delta u + d_{2}  div (u \nabla v) + u(1-u) , & x \in \Omega, \ t > 0, \\
			v_{t} = d_{1} \Delta v + \frac{1}{\tau} (w - v), & x \in \Omega, \ t > 0, \\
			w_{t} = d_{1} \Delta w + \frac{1}{\tau} (a - w), & x \in \Omega, \ t > 0, \\
			a_{t} = \frac{2\rho u^2}{1+u^2} - (\mu + \beta u) a, & x \in \Omega, \ t > 0, \\
			\frac{\partial u}{\partial \vec{n}} = \frac{\partial v}{\partial \vec{n}} = \frac{\partial w}{\partial \vec{n}} =\frac{\partial a}{\partial \vec{n}}= 0, & x \in \partial \Omega, \ t > 0.
		\end{cases}
	\end{equation}
	There exists a constant steady state $\left( 1,\frac{\rho}{\mu+\beta},\frac{\rho}{\mu+\beta},\frac{\rho}{\mu+\beta}\right) $.
	
	From Theorem \ref{3.9} and Theorem \ref{th3.7}, the steady-state bifurcation points are
 \begin{equation}\label{eq:4.5} d_2=\tilde{d}_{2,n}^S=-\frac{2(d_1\lambda_n+1)(d_1\lambda_n\tau+1)^2}{\lambda_n}<0,\, n=1,2,...,
	\end{equation}
	and the Hopf bifurcation points are
	\begin{figure}[!t]\label{Fig.3}
		\centering
		\includegraphics[width=0.55\textwidth]{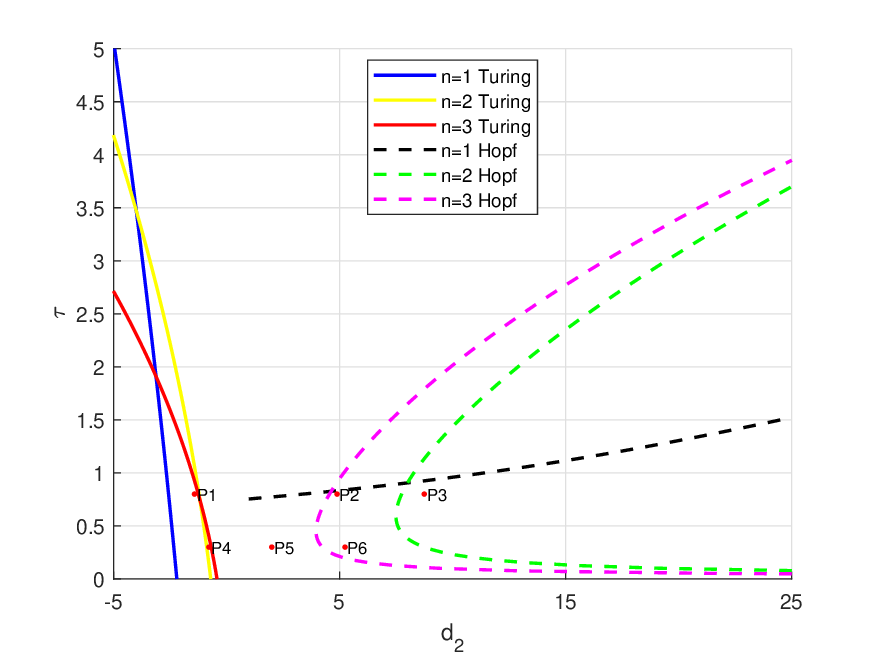}
		\caption{The bifurcation diagram of system \eqref{eq:4.1} for the strong kernel in the $(d_2, \tau)$-plane.  The  Turing bifurcation curves $d_2=\tilde{d}_{2,n}^S$ defined in \eqref{eq:4.5} and the  Hopf bifurcation curves $d_2=\tilde{d}_{2,n}^H$ defined in \eqref{eq:4.6} are plotted for $n=1, 2, 3.$ Six red points  $P_1 (-1.42, 0.8)$, $P_2 (4.89, 0.8)$, $P_3 (8.75, 0.8)$, $P_4 (-0.78, 0.3)$, $P_5 (2, 0.3)$, and $P_6 (5.24, 0.3)$ are chosen for simulation.}
		\label{Fig.3}
	\end{figure}
	
	\begin{equation}\label{eq:4.6}
		\begin{aligned}
			d_2=\tilde d_{2,n}^H=&-\frac{1+d_1\lambda_n\tau}{\lambda_n^3\tau^2}[d_1\lambda_n^2\left(1+d_1\lambda_n\tau\right)^2\left(-1+d_1+(-3+d_1)d_1\lambda_n\tau\right)+4(1+3d_1\lambda_n\tau)\\
			&(1+\lambda_n\tau(-2+4d_1+3(-1+d_1)d_1\lambda_n\tau))+(-1-4\tau-d_1\lambda_n\tau)(-((1+(-2\\
			&+d_1)\lambda_n\tau)(1+d_1\lambda_n\tau))-2\tau(2+(-1+2d_1)\lambda_n\tau))+2\lambda_n(1+d_1\lambda_n\tau)(-1+d_1\\
			&(2+\lambda_n\tau(-8+8d_1+\lambda_n\tau+6(-2+d_1)d_1\lambda_n\tau)))-(-\lambda_n(1+d_1\lambda_n\tau)^2(-1+2\\
			&d_1+d_1(-5+2d_1)\lambda_n\tau)+4\tau(-4+\lambda_n\tau(5-16d_1+3(3-4d_1)d_1\lambda_n\tau))-2(1\\
			&+d_1\lambda_n\tau)(2+\lambda_n\tau(-6+\lambda_n\tau+2d_1(6+(-8+5d_1)\lambda_n\tau))))],\ n=1,2,....
		\end{aligned}
	\end{equation}
	
	\begin{figure}[!b]\label{Fig.4}
		\centering
		\includegraphics[width=1\linewidth]{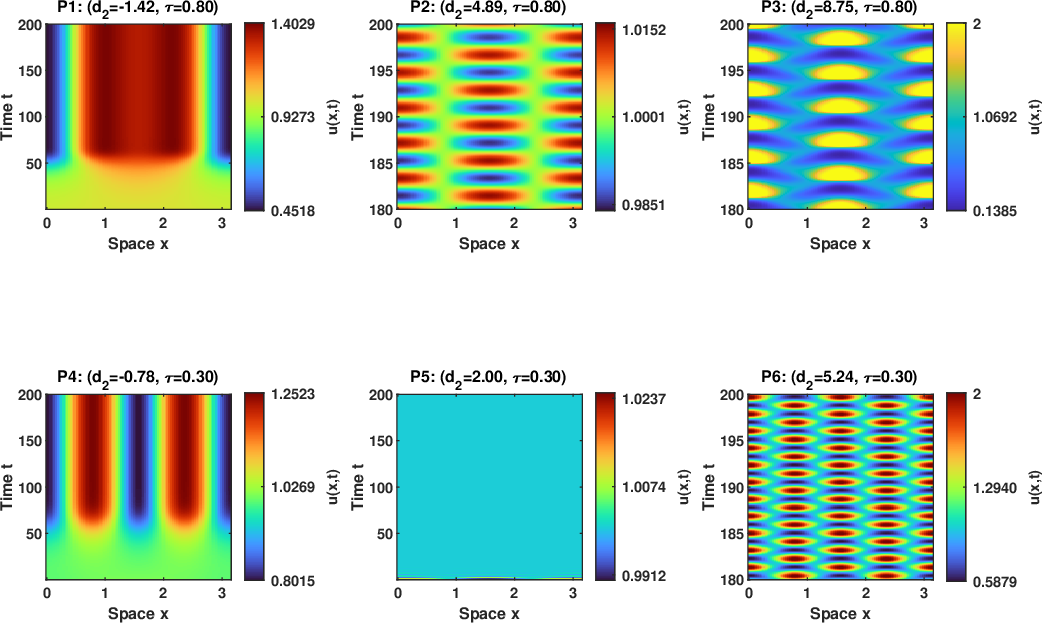}
		\caption{This illustrates the spatiotemporal plots of the biological population density $u(x,t)$ obtained from numerical simulations under the six parameter sets described in Figure \ref{Fig.3}. The horizontal axis represents the spatial position $x$, the vertical axis represents time $t$, and the color indicates the variation in the amplitude of $u(x,t)$.}
		\label{Fig.4}
	\end{figure}
	
	Figure \ref{Fig.3} shows the bifurcation structure of system (\ref{eq:4.4}) with respect to the parameter pair $(d_{2}, \tau)$ under the fixed parameter $d_{1} = 0.1$ and spatial domain $\Omega = (0, \pi)$. To investigate the effect of the time delay $\tau$ and diffusion rate $d_2$ on the spatial distribution patterns of the population, we conducted numerical simulations by fixing two distinct values of $\tau$, obtaining spatiotemporal evolution plots of the population density $u(x,t)$ as shown in Figure \ref{Fig.4}.
	
	For $\tau = 0.8$, observations along the direction of increasing $d_{2}$ reveal the following (see the first line of Figure \ref{Fig.4}): Point $P_{1}$ lies within the Turing unstable region. Numerical simulations indicate that the system produces a stationary spatially heterogeneous pattern, where the population density forms a stable, time-invariant periodic distribution in space. Points $P_{2}$ and $P_{3}$ are located within the Hopf unstable region. Here, the system exhibits coupled temporal periodic oscillations and spatial modulation, i.e., spatiotemporal oscillatory patterns. As $d_{2}$ increases from $P_{2}$ to $P_{3}$, the spatial wavelength of the oscillations increases significantly, reflecting the regulatory role of the diffusion coefficient on the spatial scale of the patterns.
	
	For $\tau = 0.3$ (see the second line of Figure \ref{Fig.4}), point $P_{4}$ is also within the Turing unstable region. The system develops a stationary spatially heterogeneous pattern, but its spatial structure differs from that of $P_{1}$ due to the different values of $\tau$ and $d_{2}$, manifesting as a multi-peak, multi-valley pattern corresponding to mode $n=3$. This indicates that the time delay influences the mode selection of Turing patterns. Point $P_{5}$ lies within the stable region, the population density eventually becomes uniformly distributed without forming any spatial structure. Further increasing $d_{2}$ to near the first Hopf bifurcation curve for mode $n=3$, point $P_{6}$ is situated within the Hopf unstable region. The system exhibits spatiotemporal oscillations, with an oscillation frequency significantly higher than that for $\tau = 0.8$, as evidenced by the increased density of stripes along the time axis in the spatiotemporal plot.


	\section{Conclusion}
	\ \ \ \
We have analyzed a reaction-diffusion system incorporating a spatiotemporal memory kernel and an independent cognitive-map variable. By introducing auxiliary variables we transformed the distributed delay system into an equivalent delay-free system, enabling classical stability and bifurcation analysis. Explicit conditions for Turing (steady-state) and Hopf bifurcations were derived for both weak (exponentially decaying) and strong (unimodal) memory kernels.

	In the case where the temporal kernel is the weak kernel $g_0(t)=\frac{1}{\tau}e^{-\frac{t}{\tau}}$, we employ Fourier modal decomposition and eigenvalue analysis to derive the stability criterion and instability threshold for the constant equilibrium point. We obtain sufficient conditions for the local stability of the constant equilibrium and further determine the critical values of the corresponding steady-state and Hopf bifurcations as the parameter $d_2$ varies. Specifically, for any memory delay $\tau \geq 0$, there exist two infinite sequences $\{d_{2,n}^S\}_{n=1}^\infty$ and $\{d_{2,n}^H\}_{n=1}^\infty$ such that the system undergoes steady-state bifurcations at $d_2 = d_{2,n}^S$ and Hopf bifurcations at $d_2 = d_{2,n}^H$.
	Numerical simulations show that, for a fixed $\tau$, spatial pattern formations emerge when $d_2$ acroses the  steady-state bifurcation curves $d_2 = d_{2,n}^S(\tau)$, and spatiotemporal patterns emerge  when $d_2$ acroses the Hopf bifurcation curves $d_2 = d_{2,n}^H(\tau)$.
	
	In the case where the temporal kernel is the strong kernel $g_1(t)=\frac{t}{\tau^2}e^{-\frac{t}{\tau}}$,
	by similarly applying Fourier modal decomposition and eigenvalue analysis, we obtain the stability criterion and instability threshold for the steady-state solution. Through the corresponding stability switching conditions and verification of the transversality conditions, we identify the steady-state bifurcation points $d_2=\tilde d_{2,n}^S$ and Hopf bifurcation points $d_2=\tilde d_{2,n}^H$ of the system.  Compared with the weak kernel case, the strong kernel not only modifies the spectral structure of the system but also yields a more complex bifurcation region in the $(d_2, \tau)$ parameter space, leading to greater diversity in steady states, periodic solutions, and spatial patterns.
	
	Compared with the model in \cite{SSW2021distributed}, where the memory term
	of the continuous-time integral kernel is described by its own population density $u(x,t)$, our formulation introduces an auxiliary cognitive map variable $a(x,t)$ governed by an ordinary differential equation. This added degree of freedom allows the model to capture not only the decay of memory but also its formation rate and erasure upon revisiting (via the parameters $\mu$ and $\beta$). As a result, our system  under a weak temporal kernel or a strong temporal kernel  can exhibit Hopf bifurcations and steady state bifurcations, whereas the model of \cite{SSW2021distributed} only has steady state bifurcations when it is under a weak temporal kernel. This implies that  a dynamic cognitive map introduces sufficient flexibility to generate both steady state bifurcations and Hopf bifurcations across a broader range of temporal kernels, leading to spatially non-homogeneous time-periodic patterns.

\appendix
\section*{Appendix A}
\textbf{Proof of Proposition \ref{p1}.}
\begin{proof}
   Assume that $(u,a)$ is the solution of (\ref{eq:1.4}). Define $v(x,t)$ as in \eqref{v0}. Then according to Lemma 2.1 and 2.2 in \cite{Zuo2021dengjia}, the function $v$ satisfies the parabolic equation
    	\begin{equation}\label{eq:2.3}
    		\begin{cases}
    			 v_{t} = d_{1} \Delta v - \frac{1}{\tau} v + \frac{1}{\tau} a, & x \in \Omega, \ t \in (-\infty, \infty), \\
    			 \frac{\partial v}{\partial \vec{n}} = 0, & x \in \partial\Omega, \ t \in (-\infty, \infty).
    		\end{cases}
    	\end{equation}
   Therefore, if $(u,a)$ is the solution of (\ref{eq:1.4}), then
    	 $(u,v,a)$ is the solution of (\ref{eq:2.2}).
    	
    	Next, we shall prove that if $(u,v,a)$ is the solution of (\ref{eq:2.2}), then $(u,a)$ is the solution of (\ref{eq:1.4}). Define $\xi(x,0) := \int_{-\infty}^{0} \int_{\Omega} G(x,y,-s)  g_{0}(-s) \eta(y,s)  dy ds.$ According to Proposition 2.1 of \cite{LiuWZ2025dengjiaxitong}, we rewrite the nonlinear function $d_{2}  div  (u \nabla v) + f(u)$ in the following form
    	\[ \begin{aligned}
    		d_{2}  div (u \nabla v) + f(u)=& d_{2}  div\left( u \nabla \int_{-\infty}^{t} \int_{\Omega} G(x,y,t-s) g_{0}(t-s) a(y,s)  dy ds \right) + f(u)  \\
    		=& d_{2}  div \left( u \nabla  \int_{-\infty}^{0} \int_{\Omega} G(x,y,t-s) g_{0}(t-s) \eta(y,s)  dy  ds  \right)  \\
    		&+ d_{2}  div  \left( u \nabla \int_{0}^{t} \int_{\Omega} G(x,y,t-s) g_{0}(t-s) a(y,s)  dy  ds  \right) + f(u) \\
    		=:& d_{2}  div  \left( u \nabla \xi(x,t) \right)
    		+f_1(u),
    	\end{aligned} \]
    	where \[ \xi(x,t): = \int_{\Omega} G(x,y,t) e^{-\frac{t}{\tau}} \, \xi(x,0)  dy
    	+ \int_{0}^{t} \int_{\Omega} G(x,y,t-s) \, g_{0}(t-s) \, a(y,s)  dy ds ,\] and
    	 \[ \begin{aligned}
    	 	f_1(u):=&d_2  div\left(u \nabla \left[-\int_{\Omega}G(x,y,t)e^{-\frac{t}{\tau}}\xi(x,0)  dy \right. \right.\\
    	 	&\left. \left.+\int_{-\infty}^{0} \int_{\Omega} G(x,y,t-s) g_{0}(t-s) \eta(y,s)dyds \right]\right) + f(u).
    \end{aligned} \]
    	
    	By Lemma 2.1 in \cite{Zuo2021dengjia},  $\xi(x,t)$ is the solution of the following equation
    	\begin{equation}\label{eq:2.4}
    		\begin{cases}
    			 \xi_{t} = d_{1} \Delta \xi - \frac{1}{\tau} \xi + \frac{1}{\tau} a, & x \in \Omega, \ t > 0, \\
    			 \frac{\partial \xi}{\partial \vec{n}} = 0, & x \in \partial \Omega, \ t > 0, \\
    			 \xi(x,0) = \xi_{0}(x), & x \in \Omega.
    		\end{cases}
    	\end{equation}
    	Thus, if $(u,\xi,a)$ is the solution of
    	\begin{equation}\label{eq:2.5}
    		\begin{cases}
    			 u_{t} = d_{1} \Delta u + d_{2}  div \big( u \nabla \xi \big) + f_1(u) , & x \in \Omega, \ t > 0, \\
    		     \xi_{t} = d_{1} \Delta \xi + \frac{1}{\tau}(a - \xi), & x \in \Omega, \ t > 0, \\
    			 a_{t} = h(u) - (\mu + \beta u) a, & x \in \Omega, \ t > 0, \\
    			 \frac{\partial u}{\partial \vec{n}} = \frac{\partial \xi}{\partial \vec{n}}=0, & x \in \partial \Omega, \ t > 0, \\
    			 a(x,t)=\eta(x,t), & x \in \Omega, \ t \in (-\infty,0], \\
    			 u(x,0) = u_{0}(x), & x \in \Omega, \\
    			 \xi(x,0) = \frac{1}{\tau} \int_{-\infty}^{0} \int_{\Omega} G(x,y,-s) \, g_{0}(-s) \, \eta(y,s)  dy  ds, & x \in \Omega,
    		\end{cases}
    	\end{equation}
    	then $(u,a)$ is the solution of (\ref{eq:1.4}).
    	
    	We now show that $\xi(x,t)$ and $v(x,t)$ are equal for $x \in \Omega$ and $t>0$. Decomposing $v(x,t)$ into two items as
    	\[ \begin{aligned}
    		v(x,t)
    		=& \int_{-\infty}^{0} \int_{\Omega} G(x,y,t-s) \, g_{0}(t-s) \, \eta(y,s) dy  ds \\
    		& + \int_{0}^{t} \int_{\Omega} G(x,y,t-s) \, g_{0}(t-s) \, a(y,s)  dy  ds,
    	\end{aligned} \]
    	Then we only need to demonstrate that the first term of $\xi(x,t)$ and $v(x,t)$ are equal. Let $	H(y,s) := e^{\frac{s}{\tau}}\eta(y,s)$, then by Definition \ref{de21} for the Green's function, the first term of $v(x,t)$ is transformed to
    	 \begin{equation}\label{eq:2.6}
    	 	\begin{aligned}
    	 	&\int_{-\infty}^{0} \int_{\Omega}  G(x,y,t-s) g_0(t-s) \eta(y,s) dyds\\
    	 	=&\frac{1}{\tau} e^{-\frac{t}{\tau}} \int_{-\infty}^{0}\int_{\Omega} G(x,y,t-s)  H(y,s)  dy  ds\\
    	 	=&: \frac{1}{\tau} e^{-\frac{t}{\tau}}J(x,t),
    	   \end{aligned}
    	 \end{equation}
    	 where
    \[ J(x,t):=\int_{-\infty}^{0} \chi(x,t-s)  ds,\]
    and $\chi(x,t-s) $ satisfies
    	 \begin{equation}\label{eq:2.7}
    	 	\begin{cases}
    	 		\mathcal{L} \chi(x,t-s) := \frac{\partial \chi(x,t-s)}{\partial t} - d_1 \Delta \chi(x,t-s) = 0, &x\in \Omega, t>s ,\\
    	 		\frac{\partial \chi(x,t-s)}{\partial t} = 0, &x \in \partial \Omega, \ t > s.
    	 	\end{cases}
    	 \end{equation}
    	 Let $z=t-s$, then $\chi(x,t-s) = \chi(x,z)$. Then  \[ J(x,t)=\int_{t}^{\infty} \chi(x,z) dz, \]
    	 and
    	 \begin{equation}\label{eq:2.8}
    	 	\begin{cases}
    	 		\mathcal{L} \chi(x,z) = \frac{\partial \chi(x,z)}{\partial z} - d_1 \Delta \chi(x,z) = 0, &x\in \Omega, z>0 ,\\
    	 		\frac{\partial \chi(x,z)}{\partial z} = 0, &x \in \partial \Omega, \ z > 0.
    	 	\end{cases}
    	 \end{equation}
    	 We can find that $J(x,t)$ is the solution of the following initial-boundary value problem
    	 \begin{equation}\label{eq:2.9}
    	 	\begin{cases}
    	 		\mathcal{L} u(x,t)= 0, & x \in \Omega, \ t > 0, \\
    	 		\frac{\partial u(x,t)}{\partial \vec{n}} = 0, & x \in \partial \Omega, \ t > 0, \\
    	 		u(x,0) = J(x,0), & x \in \Omega.
    	 	\end{cases}
    	 \end{equation}

    	 For the first term of $\xi(x,t)$, we have
    	 \begin{equation}\label{eq:2.10}
    	 \begin{aligned}
    	  \int_{\Omega} G(x,y,t) e^{-\frac{t}{\tau}} \, \xi(x,0) dy
    	 =&\frac{1}{\tau}e^{-\frac{t}{\tau}} \,\int_{\Omega} G(\zeta,x,t) \,  \int_{-\infty}^{0} \int_{\Omega}G(x,y,-s)  e^{\frac{s}{\tau}} \, \eta(y,s)  dy  ds dx\\
    	 =&\frac{1}{\tau}e^{-\frac{t}{\tau}} \,\int_{\Omega} G(\zeta,x,t) \,  \int_{-\infty}^{0} \int_{\Omega}G(x,y,-s)  H(y,s) dy  ds dx\\
    	 =& \frac{1}{\tau}e^{-\frac{t}{\tau}} \int_{\Omega} G(\zeta,x,t) \int_{-\infty}^{0} \chi(x,-s)  ds  dx \\
    	 =&:\frac{1}{\tau}e^{-\frac{t}{\tau}} Q(\zeta,t),
    	\end{aligned}
    	 \end{equation}
    where
 \[  Q(\zeta,t)=  \int_{\Omega} G(\zeta,x,t) J(x,0)  dx.\]

    	 By the fundamental properties of the solution to the initial value problem for the heat equation\cite{evans2010pde}, we know that $\displaystyle\lim_{t \to 0} Q(\zeta,t)=J(\zeta,0)$. Then $Q(x,t)$ is also the solution of (\ref{eq:2.9}).
    	 By the energy method in \cite{evans2010pde}, from the uniqueness of the solution to (\ref{eq:2.9}), we have
    	 $J(x,t) = Q(x,t), \, x\in \Omega, t>0. $
    	Therefore, the first term of $\xi(x,t)$ and $v(x,t)$ are equal, which in turn results in
    \begin{equation}\label{xi}
    	\xi(x,t) = v(x,t), \ x \in \Omega, \ t > 0.
    	\end{equation}

    	 On the other hand, a similar calculation shows that
    	 \[-\int_{\Omega}G(x,y,t)e^{-\frac{t}{\tau}}\xi(x,0)  dy +\int_{-\infty}^{0} \int_{\Omega} G(x,y,t-s) g_{0}(t-s) \eta(y,s)dyds=0,\]
    	 which leads to  $f_1(u)=f(u)$. Together with  \eqref{xi}, this implies that $(u,v,a)$ solves (\ref{eq:2.2}) if and only if $(u,\xi,a)$ solves
    	\eqref{eq:2.5}. Since a solution $(u,\xi,a)$ of
    	\eqref{eq:2.5} corresponds the solution $(u,a)$ of (\ref{eq:1.4}), it follows that a solution  $(u,v,a)$ to (\ref{eq:2.2}) also yields a solution $(u,a)$  to (\ref{eq:1.4}). This completes the proof.
\end{proof}

\section*{Appendix B}
\textbf{Proof of Lemma \ref{lem3.1}.}

\begin{proof}
		Differentiating (\ref{eq:3.8}) gives
	\[ \begin{aligned}
			\frac{d}{dp} d_{2}^{S}(p)&=\frac{h_{2}}{h_{1}} \cdot \left(\frac{\left[d_{1}\left(d_{1} p \tau+1\right)+\left(d_{1} p-f^{\prime}(1)\right) d_{1} \tau\right] p-\left(d_{1} p-f^{\prime}(1)\right)\left(d_{1} p \tau+1\right)}{p^{2}}\right) \\
			&=\frac{h_{2}}{h_{1}} \cdot \left(\frac{d_{1}^{2} \tau p^{2}+f^{\prime}(1)}{p^{2}}\right).
		\end{aligned} \]
		Since $h_1>0$, $h_2<0$, the sign of the derivative is opposite to that of $F(p):=d_{1}^{2}\tau p^{2}+f^{\prime}(1)$. The quadratic $F(p)$ possesses a unique positive zero, denoted by $p_*$. Consequently, $F(p)<0$ for $p \in (0,p_*)$ and $F(p)>0$ for $p \in (p_*,+\infty)$. Thus, $d_2^S(p)$ is  increasing on $(0,p_*)$ and decreasing on $(p_*,+\infty)$. The limits as
		$p\to 0^+$ and $p\to +\infty$ follow directly from (\ref{eq:3.8}).
		
		For part $(ii)$, we differentiate $d_2^H(p)$ derivative with respect to
		$p$ yielding
		\[ \begin{aligned}
			\frac{d}{dp} d_2^{H}(p)= \frac{\tilde{F}(p)}{\tau h_1p^{2}},
		\end{aligned} \]
		where $\tilde F(p):=4 \tau^{2} d_{1}^{3} p^{3}+\left(3 \tau-4 h_{2} \tau^{2}-3 f^{\prime}(1) \tau^{2}\right) d_{1}^{2} p^{2}+\left(f^{\prime}(1)+h_{2}\right)\left[1-\left(f^{\prime}(1)+h_{2}\right) \tau+\right.\\
		\left.f^{\prime}(1) h_{2} \tau^{2}\right]$.
		Because $\tau$, $h_1$, $p^2>0$, the sign of the derivative is determined by the numerator $\tilde F(p)$.  We note that
		$ \displaystyle\lim_{p \to 0}\tilde F(p)=\left(f^{\prime}(1)+h_{2}\right)\left[1-\left(f^{\prime}(1)+h_{2}\right) \tau+f^{\prime}(1) h_{2} \tau^{2}\right]<0 $ and $\displaystyle\lim_{p \to +\infty}\tilde F(p)=+\infty$. Hence $\tilde F(p)$ must admit at least one positive root. Differentiating $\tilde F(p)$ gives
		 $ \frac{d}{dp} \tilde{F}(p)=12 \tau^{2} d_{1}^{3} p^{2}+2\left(3 \tau-4 h_2 \tau^{2}-3 f^{\prime}(1) \tau^{2}\right) d_{1}^{2} p ,$ which is positive
 for all $p\in(0,+\infty)$. Therefore $\tilde F(p)$ is strictly increasing on $(0,+\infty)$ and possesses exactly one  root, denoted by $p^*$. It follows that $\tilde F(p)<0$ for $(0,p^*)$ and $\tilde F(p)>0$ for $(p^*,+\infty)$. Consequently, $d_2^H(p)$ is decreasing on $(0,p^*)$ and increasing on $ (p^*,+\infty)$. Moreover, $\displaystyle\lim_{p \to 0}d_2^H(p)=+\infty$,$\displaystyle\lim_{p \to +\infty}d_2^H(p)=+\infty$. 
	\end{proof}

\section*{Appendix C}
The definitions $a_i, b_i, c_i, d_4, e_4, f_4, i=3,4$  in \eqref{eq:3.36}:
	  	\begin{align*}
	  		a_3 =& \frac{16}{\tau} - 8 f^\prime(1) - 9 h_2,\\
	  	    b_3 =& \frac{10}{\tau^2} - \frac{12 f^\prime(1)}{\tau} -\frac{12h_2}{\tau}+ 6 f^\prime(1) h_2 + 2(f^\prime(1))^2 + 3h_2^2,\\
	  		c_3 =& \frac{4 h_2 f^\prime(1)}{\tau} -  \frac{4 f^\prime(1)}{\tau^2} -  \frac{4 h_2}{\tau^2}+ \frac{2 (f^\prime(1))^2}{\tau} - h_2  (f^\prime(1))^2 + \frac{2 h_2^2 }{\tau} - h_2^2 f^\prime(1) + \frac{2}{\tau^3},\\
	  		a_4 =& \left( -16 f^\prime(1)- 24  h_2 + \frac{32 }{\tau}\right)  d_1^5, \\[4pt]
	  		b_4 =& \left( 10 (f^\prime(1))^2 + 40 f^\prime(1) h_2 + 24 h_2^2 + \frac{50 }{\tau^2} - \frac{60 f^\prime(1)}{\tau} - \frac{80  h_2}{\tau}\right) d_1^4, \\[4pt]
	  		c_4 =& -2 d_1^3 \left(f^\prime(1)\right)^3 - 22 d_1^3 (f^\prime(1))^2 h_2 - 32 d_1^3 f^\prime(1) h_2^2 - 8d_1^3h_2^3+\frac{38 d_1^3 }{\tau^3} -\frac{86d_1^3f^\prime(1)}{\tau^2} \\
	  		&-\frac{9 d_1^2 d_2 h_1}{\tau^2} - \frac{102 d_1^3 h_2}{\tau^2} + \frac{34 d_1^3 \left(f^\prime(1)\right)^2}{\tau}+\frac{116 d_1^3 f^\prime(1) h_2}{\tau}+\frac{64 d_1^3 h_2^2}{\tau}, \\[4pt]
	  		d_4 =& 4 d_1^2 (f^\prime(1))^3 h_2 + 14 d_1^2 (f^\prime(1)^2) h_2^2 + 8 d_1^2 f^\prime(1) h_2^3 + \frac{14 d_1^2 }{\tau^4} - \frac{58 d_1^2f^\prime(1)}{\tau^3}-\frac{12d_1 d_2h_1}{\tau^3}\\
	  		&-\frac{62d_1^2h_2}{\tau^3}+\frac{42d_1^2(f^\prime(1))^2}{\tau^3}+\frac{6d_1 d_2 f^\prime(1)h_1}{\tau^2}+\frac{120d_1^2f^\prime(1) h_2}{\tau^2}+\frac{6d_1d_2h_1h_2}{\tau^2}+\frac{62d_1^2h_2^2}{\tau^2}\\
	  		&-\frac{6d_1^2 (f^\prime(1))^3}{\tau}-\frac{54d_1^2(f^\prime(1))^2h_2}{\tau}-\frac{68d_1^2f^\prime(1)h_2^2}{\tau}-\frac{16d_1^2h_2^3}{\tau}, \\[4pt]
	  		e_4 =& -2 d_1 (f^\prime(1))^3 h_2^2-2d_1(f^\prime(1))^2h_2^3+\frac{2d_1}{\tau^5} - \frac{18 d_1f^\prime(1)}{\tau^4} - \frac{4 d_2 h_1}{\tau^4}-\frac{18d_1h_2}{\tau^4}+\frac{22d_1(f^\prime(1))^2}{\tau^3}\\
	  		&+\frac{4d_2f^\prime(1)h_1}{\tau^3}+\frac{52d_1f^\prime(1)h_2}{\tau^3}+\frac{4d_2h_1h_2}{\tau^3}+\frac{26d_1 h_2^2}{\tau^3}-\frac{6d_1(f^\prime(1))^3}{\tau^2}-\frac{d_2(f^\prime(1))^2h_1}{\tau^2}-\\
	  		&\frac{42d_1(f^\prime(1))^2h_2}{\tau^2}-\frac{2d_2f^\prime(1)h_1h_2}{\tau^2}-\frac{46d_1f^\prime(1)h_2^2}{\tau^2}-\frac{d_2h_1h_2^2}{\tau^2}-\frac{10d_1h_2^3}{\tau^2}+\frac{8d_1(f^\prime(1))^3h_2}{\tau}+\\
	  		&\frac{22d_1(f^\prime(1))^2h_2^2}{\tau} +\frac{12d_1f^\prime(1)h_2^3}{\tau},\\
	  		f_4 =& -\frac{2f^\prime(1)}{\tau^5}-\frac{2h_2}{\tau^5}+\frac{4(f^\prime(1))^2}{\tau^4}+\frac{8f^\prime(1)h_2}{\tau^4}+\frac{4h_2^2}{\tau^4}-\frac{2(f^\prime(1))^3}{\tau^3}-\frac{10(f^\prime(1))^2h_2}{\tau^3}\\
	  		&-\frac{10f^\prime(1)h_2^2}{\tau^3}-\frac{2h_2^3}{\tau^3}+\frac{4(f^\prime(1))^3h_2}{\tau^2}+\frac{8(f^\prime(1))^2h_2^2}{\tau^2}+\frac{4f^\prime(1)h_2^3}{\tau^2}-\frac{2(f^\prime(1))^3h_2^2}{\tau}\\
	  		&-\frac{2(f^\prime(1))^2h_2^3}{\tau},
	  	\end{align*}

\end{document}